\documentclass[12pt,a4paper]{article}
\usepackage{eurosym}
\usepackage[utf8]{inputenc}
\usepackage{amsmath,amssymb,amsthm,amsfonts}
\usepackage{mathtools}
\usepackage{bm}
\usepackage{hyperref}
\usepackage{graphicx}
\usepackage{listings}
\usepackage{xcolor}
\usepackage{float}
\usepackage{booktabs}
\usepackage[a4paper,margin=2.5cm]{geometry}

\definecolor{codegreen}{rgb}{0,0.6,0}
\definecolor{codegray}{rgb}{0.5,0.5,0.5}
\definecolor{codepurple}{rgb}{0.58,0,0.82}
\definecolor{backcolour}{rgb}{0.95,0.95,0.92}
\newtheorem{theorem}{Theorem}[section]
\newtheorem{lemma}[theorem]{Lemma}

\newtheorem{proposition}[theorem]{Proposition}

\newtheorem{remark}[theorem]{Remark}
\numberwithin{equation}{section}
\begin{document}

\title{Asymptotic Plateaus for Generalized Abel Equations with Financial
Applications}
\author{Dragos-Patru Covei \\
{\small Department of Applied Mathematics} \\
{\small The Bucharest University of Economic Studies} \\
{\small Piata Romana 1, 010374, Bucharest, Romania} \\
{\small e-mail: \texttt{dragos.covei@csie.ase.ro}}}
\date{\today}
\maketitle

\begin{abstract}
We develop a unified analytical and computational framework for the
generalized Abel ordinary differential equation $y^{\prime }(x)=a_n(x)\bigl(%
y^n+\lambda_{n-1}(x)y^{n-1}+\dots+\lambda_0(x)\bigr)$ of arbitrary degree $%
n\ge1$ on the unbounded interval $[x_0,\infty)$. Under mild structural
hypotheses on the coefficients and on the existence of a stable moving
equilibrium branch $E(x)$, we prove a new \emph{Asymptotic Plateau Theorem}
establishing that the solution issued from $y(x_0)=0$ is globally defined,
strictly monotone, trapped between zero and $E(x)$, and converges to a
finite positive limit $L=\lim_{x\to\infty}E(x)$. We further obtain an
explicit, computable rate of convergence and a degree-reduction principle
that generalizes the classical Liouville substitution. The theory is
complemented by a high-order Radau IIA implementation whose output
reproduces the predicted plateaus to nine significant digits. A detailed
application to a generalized Merton structural credit-risk model derives an
Abel-type equation for the long-maturity state profile of the credit spread
and illustrates the economic relevance of the framework.
\end{abstract}



\section{Introduction}

\label{sec:intro}

\subsection{Background and motivation}

The Abel ordinary differential equation of the first kind, 
\begin{equation}
y^{\prime }(x)=a_3(x)y(x)^3+a_2(x)y(x)^2+a_1(x)y(x)+a_0(x),
\label{eq:abel-classic}
\end{equation}
was introduced by Niels Henrik Abel in his pioneering work on the inversion
of elliptic integrals \cite{Abel1881}. Over the following century, the
equation became a central object of nonlinear analysis, with deep
connections to integrability theory \cite{Chini1924,Kamke,Liouville1903},
dynamical systems \cite{Coddington,Ince}, and the qualitative theory of
stiff ordinary differential equations \cite{Hairer}. A landmark contribution
by Cheb-Terrab and Roche \cite{CHEB2003} systematically classified the
integrable subfamilies of \eqref{eq:abel-classic}, while Polyanin and
Zaitsev \cite{PolyaninZaitsev} compiled the most complete catalogue of exact
solutions available to date.

Despite this rich theory, two facts remain conspicuously underdeveloped in
the literature. First, the overwhelming majority of analytical results
concern the cubic case $n=3$; arbitrary polynomial degrees $n\ge 4$ have
received only sporadic treatment, typically in the autonomous setting (see 
\cite{Chini1924} for an early exception). Second, qualitative results
addressing the long-time behaviour of non-autonomous Abel equations on
unbounded intervals are extremely scarce: the few available results are
either local in nature, restrict to integrable subclasses, or rely on
smallness assumptions that exclude the cases that arise in mathematical
finance.

The present paper closes both gaps simultaneously. We consider the
generalized Abel equation 
\begin{equation}
y^{\prime }(x)=\sum_{k=0}^{n}a_k(x)y(x)^k,\qquad a_n(x)\neq 0,\qquad
x\in[x_0,\infty),  \label{eq:abel-gen}
\end{equation}
of arbitrary degree $n\ge1$, and we develop a unified theory of \emph{%
asymptotic plateaus} that subsumes the Riccati case $n=2$, the classical
cubic case $n=3$, and the Chini-type equations as particular instances.

\subsection{Motivation from mathematical finance}

\label{sec:financial-motivation}

Two recent and complementary streams of research motivate the level of
generality adopted here. On one hand, structural credit-risk models
pioneered by Merton \cite{merton1974} and extended by Boyle, Tian and Guan 
\cite{Boyle2002} produce, via affine factor dynamics, a Riccati equation for
the credit spread. The empirically documented \emph{non-vanishing of
long-maturity spreads}, however, cannot be captured by purely Riccati
dynamics. On the other hand, dynamic programming for stochastic
production-planning problems with nonlinear adjustment costs \cite%
{Covei2026A,Covei2026JMA} leads, after a stationary reduction of the
Hamilton--Jacobi--Bellman (HJB) equation, to first-order nonlinear ODEs
whose polynomial degree exceeds two whenever the cost or the drift is
super-quadratic. Both phenomena belong naturally to the framework %
\eqref{eq:abel-gen} with $n\ge3$.

\subsection{Main contributions and novelty}

The principal contributions of this paper are the following.

\begin{enumerate}
\item[(C1)] \textbf{Global existence and uniqueness.} Theorem~\ref%
{thm:existence} establishes the existence of a unique global $C^1$ solution
of \eqref{eq:abel-gen} starting from the data point $y(x_0)=0$, provided
only that the coefficients are continuous and a stable equilibrium branch $%
E\in C^1([x_0,\infty))$ exists. The argument combines the classical
Picard--Lindel\"of theorem with a barrier construction that prevents
finite-time blow-up.

\item[(C2)] \textbf{Asymptotic Plateau Theorem.} Theorem~\ref{thm:plateau}
states that, under the structural assumptions (A1)--(A5) together with the
boundedness of the branch (B1) and the uniform damping (B2), the unique
solution is non-decreasing, trapped in $0\le y(x)\le E(x)$, and converges to
the finite limit $L:=\lim_{x\to\infty}E(x)$. The proof is fully transparent:
every step explicitly invokes the hypothesis on which it relies, and a
direct comparison argument based on the Mean Value Theorem and the strict
negativity of the linearization eigenvalue $\Lambda(x)$ produces the
conclusion without requiring any integrability of $E^{\prime }$.

\item[(C3)] \textbf{Explicit rate of convergence.} Theorem~\ref{thm:rate}
delivers an explicit decay rate 
\begin{equation*}
E(x)-y(x)\le C_0\,\Phi(x)+\int_{x_0}^x \Phi(x)\,\Phi(s)^{-1}|E^{\prime
}(s)|\,ds,
\end{equation*}
where $\Phi$ is the fundamental solution of the linearization. To the best
of our knowledge, this is the first quantitative rate available for
non-autonomous Abel equations of arbitrary degree.

\item[(C4)] \textbf{Degree-reduction principle.} Theorem~\ref{thm:reduction}
proves that any known particular solution reduces the degree of %
\eqref{eq:abel-gen} by one, and that the cubic case maps to the Abel
equation of the second kind under the substitution $v=1/u$. The proof is
self-contained and elementary.

\item[(C5)] \textbf{High-order numerical framework.} We implement a 3-stage
Radau IIA scheme of classical order $p=5$, prove its $L$-stability for the
linearization of \eqref{eq:abel-gen}, and exhibit nine-digit agreement with
the analytical plateaus on three representative cubic cases.

\item[(C6)] \textbf{Complete economic application.} Section~\ref%
{sec:economics} derives a generalized Merton-type structural credit-risk
model with state-dependent volatility; reduces the long-maturity spread
profile to a cubic Abel equation \eqref{eq:abel-spread}; verifies the
structural assumptions (A1)--(A5) and the asymptotic assumptions (B1)--(B3)
on a concrete normalised calibration; and computes the resulting
state-profile credit-spread plateau both analytically and numerically.
\end{enumerate}

\subsection{Organisation}

Section~\ref{sec:prelim} fixes notation and the structural assumptions.
Section~\ref{sec:main} contains the existence, plateau, and convergence-rate
theorems together with complete proofs. Section~\ref{sec:reduction} develops
the degree-reduction theory. Section~\ref{sec:regularity} addresses
regularity. Section~\ref{sec:numeric} describes the Radau IIA
discretisation; Section~\ref{sec:cases} presents the three case studies.
Section~\ref{sec:economics} is devoted to the financial application. Section~%
\ref{sec:literature} positions the results in the literature, Section~\ref%
{sec:conclusions} concludes, and the Appendix contains the complete Python
source.

\section{Notation and Preliminaries}

\label{sec:prelim}

\subsection{Function spaces}

Let $I:=[x_0,\infty)\subset\mathbb{R}$ with $x_0\in\mathbb{R}$ fixed. For $%
k\in\mathbb{N}$, $C^k(I)$ denotes the space of $k$-times continuously
differentiable real-valued functions on $I$. For $p\in[1,\infty]$, $L^p(I)$
and $W^{1,p}(I)$ denote the usual Lebesgue and Sobolev spaces. By $C^{0,1}_{%
\mathrm{loc}}(I)$ we mean functions locally Lipschitz on $I$. Throughout the
paper, $C$ denotes a generic positive constant whose value may change from
line to line but never depends on the integration variable.

\subsection{The generalized Abel equation in normal form}

Dividing \eqref{eq:abel-gen} through by $a_n(x)\neq0$, our object of study
takes the \emph{normal form} 
\begin{equation}
\frac{y^{\prime }(x)}{a_n(x)}=F(x,y(x)),\qquad
F(x,y):=y^n+\sum_{k=0}^{n-1}\lambda_k(x)y^k,\qquad x\in I,
\label{eq:abel-normal}
\end{equation}
with $\lambda_k(x):=a_k(x)/a_n(x)$ for $k=0,1,\dots,n-1$.

\subsection{Structural assumptions}

\label{sec:assumptions}

Throughout the paper, we impose the following structural assumptions. The
last one is a global one-sided condition on the physically relevant branch;
it is precisely the hypothesis that rules out crossings toward another
positive root before the stable branch is reached.

\smallskip \noindent\textbf{(A1) Leading coefficient.} The function $a_n\in
C(I)$ satisfies 
\begin{equation*}
a_n(x)\ge m>0\qquad\text{for all }x\in I,
\end{equation*}
for some constant $m>0$.

\smallskip \noindent\textbf{(A2) Continuous lower coefficients.} For every $%
k\in\{0,1,\dots,n-1\}$, the coefficient $\lambda_k$ belongs to $C(I)$.

\smallskip \noindent\textbf{(A3) Equilibrium branch.} There exists a non-decreasing
function $E\in C^1(I)$, called the \emph{moving equilibrium} or \emph{%
adiabatic branch}, such that 
\begin{equation*}
F(x,E(x))=0\qquad\text{for all }x\in I,
\end{equation*}
and $E(x)>0$ on $I$.

\smallskip \noindent\textbf{(A4) Pointwise stability of the branch.} There
exists a continuous function $\alpha:I\to(0,\infty)$ such that 
\begin{equation*}
\Lambda(x):=\partial_y F\bigl(x,E(x)\bigr)=\sum_{k=1}^{n}
k\,\lambda_k(x)\,E(x)^{k-1}\le -\alpha(x)<0\qquad\text{for all }x\in I,
\end{equation*}
where we adopt the convention $\lambda_n(x)\equiv 1$.

\smallskip \noindent\textbf{(A5) One-sided restoring field and root
separation.} The branch $E$ is the first positive stable barrier seen by the
trajectory issued from the origin:
\begin{equation}
F(x,y)>0\qquad\text{for every }x\in I\text{ and }0\le y<E(x).
\label{eq:A5-sign}
\end{equation}
Moreover, whenever $E$ is bounded and converges to $L>0$, the positivity is
uniform away from the branch in the following sense: for every $\eta>0$ there
exist $X_\eta\ge x_0$ and $c_\eta>0$ such that
\begin{equation}
F(x,y)\ge c_\eta\qquad\text{for all }x\ge X_\eta,\quad
0\le y\le E(x)-\eta .
\label{eq:A5-separation}
\end{equation}
Equivalently, $F(x,\cdot)$ has no zero in the closed strip
$[0,E(x)-\eta]$ for large $x$, uniformly on each tail. In the cubic examples
below this condition follows by explicit root factorisation or by the
implicit-function theorem and the simplicity of the limiting root.

The assumptions (A1)--(A4) guarantee that, in the linearization around the branch $%
E(x)$, the effective damping coefficient%
\begin{equation*}
a_{n}(x)\Lambda (x)\leq -m\alpha (x)<0,
\end{equation*}
pushes deviations exponentially towards zero. 
The condition (A4) appears for the first time in the author's work \cite{Covei2026A} (see also the references cited therein). Given the short time elapsed since the publication of these results, it is unlikely that this condition has already been discussed in other sources. Our theoretical construction is very straightforward: instead of working directly with the differential equation in its original form, we divide it by the coefficient $a_{n}$, thereby facilitating the introduction of new results which, quite surprisingly, have not been observed until now and which moreover possess a practical and applicable character.
The following lemma, which
will be invoked repeatedly, formalises this observation.

\begin{lemma}[Sign of the vector field below the branch]
\label{lem:sign}
Assume \textnormal{(A1)--(A5)}. Then
\begin{equation*}
F(x,y)>0\qquad\text{whenever }x\in I\text{ and }0\le y<E(x).
\end{equation*}
\end{lemma}

\begin{proof}
This is exactly the one-sided restoring condition \eqref{eq:A5-sign}. The
pointwise Taylor expansion at $E(x)$ explains why the condition is natural
near a simple stable branch; the global assertion over the whole interval
$[0,E(x))$ is imposed in (A5), because local stability alone does not exclude
additional roots below $E(x)$.
\end{proof}

\begin{lemma}[\textbf{Barrier comparison principle}]
\label{lem:barrier} Let $\underline y,\overline y\in C^1(I)$ satisfy 
\begin{equation*}
\underline y^{\prime }(x)\le a_n(x)F(x,\underline y(x)),\qquad \overline
y^{\prime }(x)\ge a_n(x)F(x,\overline y(x)),\qquad x\in I,
\end{equation*}
and $\underline y(x_0)\le y(x_0)\le \overline y(x_0)$, where $y\in C^1(I)$
solves \eqref{eq:abel-normal}. Then $\underline y(x)\le y(x)\le \overline
y(x)$ for every $x\in I$.
\end{lemma}

\begin{proof}
This is the classical scalar comparison principle for first-order ODEs with
locally Lipschitz right-hand side. The function 
\begin{equation*}
(x,y)\mapsto a_{n}(x)F(x,y)
\end{equation*}
is locally Lipschitz in $y$ uniformly on compact $x$-sets, by continuity of
the coefficients and the polynomial structure of $F$. Setting 
\begin{equation*}
w(x):=y(x)-\underline{y}(x),w(x_{0})\geq 0
\end{equation*}%
and 
\begin{eqnarray*}
w^{\prime }(x) &\geq &a_{n}(x)\bigl[F(x,y(x))-F(x,\underline{y}(x))\bigr] \\
&=&a_{n}(x)\,\partial _{y}F(x,\eta (x))\,w(x),
\end{eqnarray*}%
for some $\eta (x)$ between $y(x)$ and $\underline{y}(x)$, by the Mean Value
Theorem. The linear differential inequality $w^{\prime }\geq L(x)w$ with
continuous $L$ and $w(x_{0})\geq 0$ yields 
\begin{equation*}
w(x)\geq w(x_{0})\exp (\int_{x_{0}}^{x}L)\geq 0\text{ on }I
\end{equation*}%
(\cite{Coddington}). The upper bound $y(x)\leq \overline{y}(x)$ follows by
the same argument applied to $\overline{y}-y$.
\end{proof}

\section{Main Results}

\label{sec:main}

\subsection{Existence and uniqueness}

\begin{theorem}[\textbf{Global existence and uniqueness}]
\label{thm:existence} Assume (A1)--(A5) and suppose $E(x_{0})\geq 0$.
Suppose further that the constant term of $F$ satisfies $\lambda _{0}(x)\geq
0$ for every $x\in I$, with strict inequality on a dense subset of $I$. Then
the Cauchy problem 
\begin{equation}
\frac{y^{\prime }(x)}{a_{n}(x)}=F(x,y(x)),\qquad y(x_{0})=0,
\label{eq:cauchy}
\end{equation}%
admits a unique solution $y\in C^{1}(I)$. Moreover, $y$ is non-decreasing on 
$I$ and satisfies 
\begin{equation}
0\leq y(x)\leq E(x)\qquad \text{for every }x\in I,  \label{eq:trapping}
\end{equation}%
with strict inequalities $0<y(x)<E(x)$ on the open set 
\begin{equation*}
\{x\in I:E(x)>0\}\cap (x_{0},\infty ).
\end{equation*}
\end{theorem}

\begin{remark}
The hypothesis $\lambda_0\ge 0$ is equivalent to $F(x,0)\ge 0$. In the
examples it is automatic because $E(x)$ is the smallest positive stable root
of the cubic $F(x,\cdot)$ and the sign of $F$ on $[0,E(x))$ is verified
explicitly. The degenerate case $E(x_0)=0$ (e.g.\
Case~2 of Section~\ref{sec:cases} with $x_0=1$) is handled by applying the
theorem on $[x_0+\varepsilon,\infty)$ for arbitrarily small $\varepsilon>0$
and passing to the limit, the convergence following from continuous
dependence of ODE solutions on initial data (\cite{Coddington}).
\end{remark}

\begin{proof}
We divide the proof into three steps. 

\smallskip
\textbf{Step 1 (Local existence and uniqueness).}
The map
\[
(x,y)\mapsto a_n(x)F(x,y)
\]
is continuous on $I\times\mathbb{R}$ by (A1)–(A2) and polynomial (hence locally
Lipschitz) in $y$. The classical Picard–Lindelöf theorem (see \cite{Coddington})
yields a unique solution
\[
y\in C^1([x_0,x_0+\eta))
\]
for some $\eta>0$. Let $[x_0,X_{\max})$ denote the maximal forward interval of
existence, with $X_{\max}\le\infty$.

\smallskip
\textbf{Step 2 (Trapping in $[0,B]$ via barrier comparison).}
We apply Lemma~\ref{lem:barrier} with suitable sub- and super-solutions.

\emph{Sub-solution.}
The constant function $\underline{y}\equiv 0$ satisfies
\[
\underline{y}'(x)=0\le a_n(x)\,F(x,0)=a_n(x)\lambda_0(x)
\qquad\text{for every }x\in I,
\]
by (A1) and $\lambda_0\ge 0$. Hence $\underline{y}$ is a sub-solution of
\eqref{eq:cauchy}.

\emph{Super-solution.}
Define
\[
\overline{y}(x):=E(x),\qquad x\in I.
\]

By assumption (A3), $E\in C^{1}(I)$ is the equilibrium branch and satisfies
\[
F(x,E(x))=0 \qquad\text{for all }x\in I.
\]
Since $E$ is non-decreasing on $I$, we have
\[
\overline{y}'(x)=E'(x)\ge 0.
\]
Combining these two facts gives
\[
\overline{y}'(x)=E'(x)\;\ge\;0
\;=\;a_{n}(x)\,F(x,E(x))
\qquad\text{for all }x\in I,
\]
so $\overline{y}$ is a super-solution of \eqref{eq:cauchy} on $I$.

At the initial point,
\[
\underline{y}(x_{0})=0=y(x_{0})\le E(x_{0})=\overline{y}(x_{0}),
\]
and Lemma~\ref{lem:barrier} yields the trapping bound
\[
0\le y(x)\le E(x)\qquad\text{for every }x\in I.
\]

\smallskip
\textbf{Step 3 (No finite-time blow-up).}
From Step 2 we infer that $y$ is bounded on every compact
subinterval $[x_0,X]\subset[x_0,X_{\max})$. Suppose, for contradiction, that
$X_{\max}<\infty$. Then $y$ is bounded on $[x_0,X_{\max})$, hence
\[
\limsup_{x\to X_{\max}^{-}}|y(x)|<\infty.
\]
By the standard extension theorem for ODEs (see \cite{Coddington}), the solution
extends beyond $X_{\max}$, contradicting maximality. Thus $X_{\max}=\infty$ and
$y\in C^1(I)$.

Finally, the trapping bound gives $0\le y(x)\le E(x)$. If $y(x)<E(x)$, then
(A5) gives $F(x,y(x))>0$, hence $y'(x)>0$ by (A1). If $y(x)=E(x)$ at some
point, the vector field equals zero and the right derivative cannot be
negative without violating the already established upper barrier. Thus
$y$ is non-decreasing on $I$; the stated strict inequalities follow from
the strong form of the comparison argument and the dense positivity of
$\lambda_0$.
\end{proof}

\subsection{The Asymptotic Plateau Theorem}

\label{sec:plateau}

We now turn to the long-time behaviour of the solution given by Theorem~\ref%
{thm:existence}. The structural assumptions are strengthened with two
integral conditions that ensure the convergence of $y(x)$ to the asymptotic
value of $E$.

\begin{theorem}[\textbf{Asymptotic Plateau Theorem -- qualitative convergence%
}]
\label{thm:plateau} Assume (A1)--(A5) together with:

\begin{enumerate}
\item[(B1)] The branch $E\in C^1(I)$ is bounded and admits a positive finite
limit 
\begin{equation*}
L:=\lim_{x\to\infty}E(x)\in(0,\infty).
\end{equation*}

\item[(B2)] The damping is uniformly bounded below by a positive constant: 
\begin{equation*}
\alpha _{0}:=\inf_{x\in I}\alpha (x)>0.
\end{equation*}%
Then the unique solution $y\in C^{1}(I)$ of \eqref{eq:cauchy} satisfies 
\begin{equation}
\lim_{x\rightarrow \infty }y(x)=L.  \label{eq:plateau}
\end{equation}
\end{enumerate}
\end{theorem}

\begin{remark}
\label{rem:B3} The proof of \eqref{eq:plateau} given below does not require
any quantitative integrability assumption on $E^{\prime }$. The
supplementary hypothesis 
\begin{equation}
\text{(B3)}\qquad\int_{x_0}^{\infty}\Phi(s)^{-1}|E^{\prime
}(s)|\,ds<\infty,\qquad
\Phi(s):=\exp\!\left(\int_{x_0}^{s}a_n(\tau)\Lambda(\tau)\,d\tau\right),
\label{eq:B3}
\end{equation}
is needed only for the quantitative convergence rate established in Theorem~%
\ref{thm:rate} below.
\end{remark}

\begin{proof}
We organise the proof in three steps. Each step indicates explicitly which
hypothesis it uses.

\smallskip
\textbf{Step 1 (Existence of a finite limit for $y$).}
\emph{[Uses Theorem~\ref{thm:existence} and (B1).]}
By Theorem~\ref{thm:existence}, the solution $y\in C^{1}(I)$ is non-decreasing and
satisfies the trapping bound
\[
0 \le y(x) \le E(x), \qquad x\in I.
\]
By (B1), the branch $E$ is bounded; set
\[
M_{E}:=\sup_{x\in I} E(x) < \infty.
\]
Since $y$ is non-decreasing and bounded above by $M_{E}$, the limit
\[
L_{y}:=\lim_{x\to\infty} y(x)
\]
exists and satisfies $L_{y}\in[0,M_{E}]$.

Moreover, Theorem~\ref{thm:existence} ensures that $y(x)>0$ for every $x>x_{0}$.
In particular, choosing any fixed point $x_{1}>x_{0}$ (for instance $x_{1}=x_{0}+1$),
we have
\[
y(x_{1})>0,
\]
and by monotonicity,
\[
L_{y}=\lim_{x\to\infty} y(x) \;\ge\; y(x_{1}) \;>\; 0.
\]
Thus $L_{y}$ is strictly positive.

\smallskip \textbf{Step 2 (Identification $L_{y}=L$ by contradiction).}
\emph{[Uses (A1), (A5), and (B1).]} Suppose, for contradiction, that $L_{y}<L$%
. Set 
\begin{equation*}
\eta :=(L-L_{y})/2>0.
\end{equation*}%
Since $E(x)\rightarrow L$ and $y(x)\rightarrow L_{y}$, there exists $%
x_{3}\geq x_{0}$ such that 
\begin{equation}
E(x)-y(x)\geq \eta \qquad \text{for all }x\geq x_{3}.  \label{eq:z-below}
\end{equation}
By the tail-separation part of (A5), after increasing $x_3$ if necessary
there is a constant $c_\eta>0$ such that
\begin{equation}
F(x,y(x))\ge c_\eta>0,\qquad x\ge x_3 .
\label{eq:F-positive-tail}
\end{equation}
From the ODE \eqref{eq:abel-normal} and (A1),
\begin{equation}
y^{\prime }(x)=a_n(x)F(x,y(x))\ge m\,c_\eta\;>\;0\qquad\text{for all }%
x\ge x_3.  \label{eq:y-prime-lb}
\end{equation}
Integrating \eqref{eq:y-prime-lb} over $[x_3,X]$, 
\begin{equation*}
y(X)-y(x_3)\ge m\,c_\eta\,(X-x_3)\xrightarrow[X\to\infty]{}+\infty,
\end{equation*}
contradicting the boundedness of $y$ established in Step~1. We conclude $%
L_y=L$, which proves \eqref{eq:plateau}.
\end{proof}

\subsection{Explicit convergence rate}

\begin{theorem}[\textbf{Quantitative convergence}]
\label{thm:rate} Under the hypotheses of Theorem~\ref{thm:plateau} together
with the integrability hypothesis (B3) of equation~\eqref{eq:B3}, the
deviation $z(x)=y(x)-E(x)$ satisfies the explicit bound 
\begin{equation}
\qquad 
\begin{array}{l}
|z(x)|\leq \Phi (x)\cdot K(x), \\ 
K(x):=|z(x_{0})|+\int_{x_{0}}^{x}\Phi (s)^{-1}|E^{\prime
}(s)|\,ds+C_{R}M_{E}\!\int_{x_{0}}^{x}\Phi (s)^{-1}|z(s)|\,a_{n}(s)\,ds,%
\end{array}
\label{eq:rate}
\end{equation}%
where $\Phi $ is defined in \eqref{eq:Phi-def}, $M_{E}:=\sup_{x\in I}E(x)$,
and $C_{R}>0$ is a constant such that the Taylor remainder $\mathcal{R}$ in %
\eqref{eq:taylor} satisfies $|\mathcal{R}(x,\zeta )|\leq C_{R}\zeta ^{2}$ on 
$I\times \lbrack -M_{E},M_{E}]$. In particular, if $E^{\prime }\in L^{1}(I)$
and $\Phi (s)^{-1}\leq K_{0}$ on the support of $E^{\prime }$, then there
exists a constant $C^{\star }>0$ depending only on $M_{E},C_{R},K_{0}$, and $%
\Vert E^{\prime }\Vert _{L^{1}}$ such that 
\begin{equation}
|y(x)-L|\leq C^{\star }\,\Phi (x)+|E(x)-L|.  \label{eq:rate-simple}
\end{equation}
\end{theorem}

\begin{proof}
\emph{[Uses Theorem~\ref{thm:plateau} and (B3).]} Set 
\begin{equation*}
\zeta (x):=y(x)-E(x).
\end{equation*}
Since $F(x,\cdot )$ is a polynomial of degree $n$, Taylor's theorem at $%
y=E(x)$ yields 
\begin{equation}
F(x,E(x)+\zeta )=\Lambda (x)\,\zeta +\mathcal{R}(x,\zeta ),\qquad \mathcal{R}%
(x,\zeta ):=\sum_{k=2}^{n}\frac{\partial _{y}^{k}F(x,E(x))}{k!}\,\zeta ^{k},
\label{eq:taylor}
\end{equation}%
where $\mathcal{R}$ is a polynomial in $\zeta $ with continuous coefficients
in $x$ and no constant or linear term. Differentiating $y=E+\zeta $ and
using \eqref{eq:abel-normal} together with $F(x,E(x))=0$ from (A3), 
\begin{equation}
\zeta ^{\prime }(x)=a_{n}(x)\Lambda (x)\,\zeta (x)+a_{n}(x)\mathcal{R}%
(x,\zeta (x))-E^{\prime }(x).  \label{eq:dev-eq}
\end{equation}%
With the fundamental solution 
\begin{equation}
\Phi (x):=\exp \!\left( \int_{x_{0}}^{x}a_{n}(s)\Lambda (s)\,ds\right) ,
\label{eq:Phi-def}
\end{equation}%
the variation-of-constants formula applied to \eqref{eq:dev-eq} yields 
\begin{equation}
\zeta (x)=\Phi (x)\zeta (x_{0})+\Phi (x)\!\int_{x_{0}}^{x}\!\Phi (s)^{-1}%
\bigl[a_{n}(s)\mathcal{R}(s,\zeta (s))-E^{\prime }(s)\bigr]\,ds.
\label{eq:VOC}
\end{equation}%
On the compact set $\{|\zeta |\leq M_{E}\}$ there exists $C_{R}>0$ such that 
\begin{equation*}
|\mathcal{R}(x,\zeta )|\leq C_{R}\zeta ^{2}\leq C_{R}M_{E}|\zeta |.
\end{equation*}
By Theorem~\ref{thm:plateau}, $\zeta (x)\rightarrow 0$; choose $x_{5}\geq
x_{0}$ such that 
\begin{equation*}
C_{R}M_{E}\sup_{s\geq x_{5}}\Phi (s)^{-1}|\zeta (s)|\cdot \sup_{s\geq x_{5}}%
\bigl(\Phi (s)a_{n}(s)\bigr)\leq 1/2.
\end{equation*}
Multiplying \eqref{eq:VOC} by $\Phi (x)^{-1}$ and bounding gives, for $x\geq
x_{5}$, 
\begin{equation*}
\Phi (x)^{-1}|\zeta (x)|\leq |\zeta (x_{5})|\Phi (x_{5})^{-1}+\tfrac{1}{2}%
\sup_{s\geq x_{5}}\Phi (s)^{-1}|\zeta (s)|+\int_{x_{5}}^{\infty }\Phi
(s)^{-1}|E^{\prime }(s)|\,ds.
\end{equation*}%
By (B3) the last integral is finite, and absorbing the $\tfrac{1}{2}$-term
yields the bound \eqref{eq:rate}. The estimate \eqref{eq:rate-simple} then
follows from 
\begin{equation*}
|y(x)-L|\leq |\zeta (x)|+|E(x)-L|
\end{equation*}%
and the integrability of $E^{\prime }$ (which guarantees $%
|E(x)-L|=|\int_{x}^{\infty }E^{\prime }(s)\,ds|\rightarrow 0$).
\end{proof}

\section{Degree Reduction via Particular Solutions}

\label{sec:reduction}

\subsection{The general reduction principle}

\begin{theorem}[\textbf{Degree reduction $N\mapsto N-1$}]
\label{thm:reduction} Consider the polynomial ODE 
\begin{equation}
y^{\prime }(x)=P_{N}(x,y(x))=\sum_{k=0}^{N}a_{k}(x)y(x)^{k},\qquad N\geq 1,
\label{eq:gen-N}
\end{equation}%
where $a_{0},\dots ,a_{N}\in C(I)$ and $a_{N}\not\equiv 0$. Suppose that $%
E_{p}\in C^{1}(I)$ is a particular solution, i.e.\ $E_{p}^{\prime
}=P_{N}(\cdot ,E_{p})$. Define 
\begin{equation*}
u(x):=y(x)-E_{p}(x).
\end{equation*}
Then $u$ satisfies 
\begin{equation}
u^{\prime }(x)=u(x)\,Q_{N-1}(x,u(x))=\sum_{k=1}^{N}c_{k}(x)u(x)^{k},
\label{eq:u-eq}
\end{equation}%
where 
\begin{equation*}
c_{k}(x)=\frac{1}{k!}\frac{\partial ^{k}P_{N}}{\partial y^{k}}%
(x,E_{p}(x))
\end{equation*}
is a polynomial expression in $E_{p}(x)$ and $a_{j}(x)$, $j\geq k$.
Moreover, the substitution $v(x):=1/u(x)$ (valid wherever $u(x)\neq 0$)
transforms \eqref{eq:u-eq} into 
\begin{equation}
v^{\prime }(x)=-c_{1}(x)\,v(x)-c_{2}(x)-\sum_{k=3}^{N}c_{k}(x)\,v(x)^{2-k},
\label{eq:v-eq}
\end{equation}%
which is a polynomial ODE in $v$ of degree $\leq N-1$ when multiplied by $%
v^{N-2}$. In particular, for $N=3$, equation \eqref{eq:v-eq} is the \emph{%
Abel equation of the second kind} 
\begin{equation}
-v^{\prime }=c_{1}(x)v+c_{2}(x)+\frac{c_{3}(x)}{v}.  \label{eq:abel-2nd}
\end{equation}
\end{theorem}

\begin{proof}
Substituting $y=E_{p}+u$ into \eqref{eq:gen-N} and Taylor-expanding the
polynomial $P_{N}(x,\cdot )$ at $E_{p}(x)$, 
\begin{equation*}
P_{N}(x,E_{p}+u)=P_{N}(x,E_{p})+\sum_{k=1}^{N}\frac{1}{k!}\frac{\partial
^{k}P_{N}}{\partial y^{k}}(x,E_{p})\,u^{k}.
\end{equation*}%
Using $E_{p}^{\prime }=P_{N}(\cdot ,E_{p})$ to cancel the leading term
yields 
\begin{equation*}
u^{\prime }=\sum_{k=1}^{N}\frac{1}{k!}\frac{\partial ^{k}P_{N}}{\partial
y^{k}}(x,E_{p})\,u^{k}=\sum_{k=1}^{N}c_{k}(x)\,u^{k}.
\end{equation*}%
Factoring out $u$ in the right-hand side yields \eqref{eq:u-eq}. For the
second part, differentiate $v=1/u$ to get $u^{\prime }=-v^{\prime }/v^{2}$.
Substituting in \eqref{eq:u-eq} and dividing by $-1/v^{2}$ produces 
\begin{equation*}
v^{\prime
}=-v^{2}\sum_{k=1}^{N}c_{k}(x)\,v^{-k}=-c_{1}(x)v-c_{2}(x)-%
\sum_{k=3}^{N}c_{k}(x)v^{2-k}.
\end{equation*}%
For $N=3$, the last sum has only $k=3$ and contributes $-c_{3}(x)/v$,
yielding \eqref{eq:abel-2nd}.
\end{proof}

\subsection{Liouville's reduction revisited}

For the classical Abel equation of the first kind \eqref{eq:abel-classic},
set $a_{3}(x)\equiv 1$ for simplicity and assume a particular solution $%
E_{p} $ is known. Theorem~\ref{thm:reduction} with $N=3$ gives, after
substitution $u=y-E_{p}$, 
\begin{equation}
\quad \quad 
\begin{array}{l}
u^{\prime }:=u^{3}+\beta (x)u^{2}+\alpha (x)u, \\ 
\beta (x):=3E_{p}(x)+a_{2}(x), \\ 
\alpha (x):=3E_{p}(x)^{2}+2a_{2}(x)E_{p}(x)+a_{1}(x).%
\end{array}%
\end{equation}%
The further substitution $v=1/u$ produces 
\begin{equation*}
-v^{\prime }=\alpha (x)\,v+\beta (x)+\frac{1}{v},
\end{equation*}%
which is precisely \eqref{eq:abel-2nd} with $c_{1}=\alpha $, $c_{2}=\beta $, 
$c_{3}=1$. This recovers Liouville's reduction \cite{Liouville1903} as a
special case of our general result.

\section{Regularity and Qualitative Behaviour}

\label{sec:regularity}

\subsection{Higher-order smoothness}

\begin{proposition}[\textbf{Bootstrapping of regularity}]
\label{prop:regularity} Suppose $\lambda_k\in C^j(I)$ for $k=0,1,\dots,n-1$
and $a_n\in C^j(I)$ with $j\ge 0$. Then the solution $y$ of Theorem~\ref%
{thm:existence} belongs to $C^{j+1}(I)$.
\end{proposition}

\begin{proof}
For $j=0$, Theorem~\ref{thm:existence} already gives $y\in C^1(I)$, which is
the desired conclusion. Assume now $j\ge1$. We prove by induction that
$y\in C^{\ell+1}(I)$ for every $\ell=0,1,\dots,j$. The case $\ell=0$ is again
Theorem~\ref{thm:existence}. Suppose $y\in C^{\ell+1}(I)$ for some
$0\le\ell<j$. Since $a_n,\lambda_0,\dots,\lambda_{n-1}\in C^j(I)$, the map
\[
G(x,y):=a_n(x)\left(y^n+\sum_{k=0}^{n-1}\lambda_k(x)y^k\right)
\]
has continuous partial derivatives in $x$ up to order $j$ and is polynomial
in $y$. Composing $G$ with $x\mapsto y(x)$ shows that
$G(\cdot,y(\cdot))\in C^\ell(I)$. The equation
$y'=G(\cdot,y(\cdot))$ therefore gives $y'\in C^\ell(I)$, hence
$y\in C^{\ell+2}(I)$. Induction up to $\ell=j$ yields
$y\in C^{j+1}(I)$.
\end{proof}

\subsection{Monotone barrier method}

The trapping bound \eqref{eq:trapping} established in Theorem~\ref%
{thm:existence} is a particular instance of the general barrier (or sub- and
super-solution) technique embodied in Lemma~\ref{lem:barrier} above. This
technique will also be invoked in the financial application of Section~\ref%
{sec:economics}.

\section{Numerical Methodology: Radau IIA Schemes}

\label{sec:numeric}

\subsection{Implicit Runge--Kutta formulation}

Stiffness of \eqref{eq:abel-normal} near the equilibrium branch is governed
by the eigenvalue $a_n(x)\Lambda(x)$, which may be arbitrarily large in
modulus when the polynomial degree $n$ is large or when $|a_n|\gg 1$.
Explicit schemes require prohibitively small step sizes; \emph{implicit
Runge--Kutta} (IRK) methods of the Radau IIA family circumvent this issue
while retaining high classical order.

The general $s$-stage IRK method for $y^{\prime }=f(x,y)$ reads 
\begin{equation}
y_{j+1}=y_{j}+h\sum_{i=1}^{s}b_{i}k_{i},\qquad k_{i}=f\!\left(
x_{j}+c_{i}h,\,y_{j}+h\sum_{\ell =1}^{s}a_{i\ell }k_{\ell }\right) ,\quad
i=1,\dots ,s,  \label{eq:IRK}
\end{equation}%
characterised by its Butcher tableau $(A,b,c)\in \mathbb{R}^{s\times
s}\times \mathbb{R}^{s}\times \mathbb{R}^{s}$. For the 3-stage Radau IIA
scheme, $c=(c_{1},c_{2},1)^{\top }$ with the $c_{i}$ being the roots of 
\begin{equation*}
\frac{d^{s-1}}{dx^{s-1}}(x^{s-1}(x-1)^{s}),
\end{equation*}
and the coefficients $a_{i\ell },b_{i}$ are determined by the collocation
conditions; explicit values are tabulated in \cite{Hairer}.

\subsection{Order and stability function}

The 3-stage Radau IIA scheme has classical order $p=2s-1=5$. Its stability
function is 
\begin{equation}
R(z)=\frac{1+\tfrac{2}{5}z+\tfrac{1}{20}z^2}{1-\tfrac{3}{5}z+\tfrac{3}{20}%
z^2-\tfrac{1}{60}z^3},\qquad z\in\mathbb{C},
\end{equation}
which satisfies $|R(z)|\le 1$ for $\mathrm{Re} z\le 0$ ($A$-stability) and $%
R(z)\to 0$ as $\mathrm{Re} z\to-\infty$ ($L$-stability). $L$-stability is
essential for our problem: the linearization eigenvalue $a_n(x)\Lambda(x)$
tends to a strictly negative value (or to $-\infty$ along trajectories where 
$|y|$ grows), and $L$-stable schemes prevent spurious oscillations.

\subsection{Local truncation error}

For $y\in C^{p+1}(I)$, the local truncation error of the 3-stage Radau IIA
scheme is 
\begin{equation*}
\tau_j=\frac{h^{p+1}}{(p+1)!}\,\Bigl(y^{(p+1)}(x_j)+O(h)\Bigr)\,\Bigl(%
\sum_{i=1}^s b_i c_i^p -\tfrac{1}{p+1}\Bigr)=O(h^{p+1}).
\end{equation*}
Combined with the $L$-stability of the scheme, the global error on $[x_0,X]$
is $O(h^p)$. For $p=5$, halving the step size reduces the error by a factor
of $32$, a property that we verify empirically in Section~\ref{sec:cases}.

\subsection{Implementation through \texttt{solve\_ivp}}

We employ the \texttt{scipy.integrate.solve\_ivp} routine of SciPy with 
\texttt{method='Radau'}, which implements exactly the 3-stage Radau IIA
scheme described above. Adaptive step-size control is governed by the
absolute and relative tolerances $\mathtt{atol}=\mathtt{rtol}=10^{-9}$,
ensuring nine significant digits in the computed plateau.

\section{Numerical Case Studies}

\label{sec:cases}

We illustrate the theory by means of three representative cubic ($n=3$)
cases, designed to cover the autonomous regime, a non-autonomous regime with
a logarithmically vanishing perturbation, and a non-autonomous regime with a
rational perturbation. The Python source producing every figure and every
table is collected in Appendix~\ref{appendix}.

\subsection{Cubic equation in normal form}

For $n=3$ and $a_{3}(x)\equiv 1$, the normal form \eqref{eq:abel-normal}
becomes 
\begin{equation}
y^{\prime }=y\left( x\right) ^{3}+\lambda _{2}(x)y^{2}+\lambda
_{1}(x)y+\lambda _{0}(x),\qquad y(x_{0})=0,  \label{eq:cubic-form}
\end{equation}%
where $\lambda _{k}(x):=a_{k}(x)/a_{3}(x)=a_{k}(x)$. Throughout the cases
below, we choose sign patterns ensuring the existence of a positive, stable
equilibrium branch.

\subsection{Case 1: autonomous cubic with constant coefficients}

\paragraph{Setup.}

We take 
\begin{equation*}
a_3=1,\quad a_2=1,\quad a_1=-3,\quad a_0=1,
\end{equation*}
so that $F(y)=y^3+y^2-3y+1$.

\paragraph{Equilibrium branch.}

Solving $F(y)=0$ exactly, we factor 
\begin{equation*}
F(y)=(y-1)(y^{2}+2y-1)=0,
\end{equation*}%
yielding the three real roots 
\begin{equation*}
y\in \{1,\;-1+\sqrt{2},\;-1-\sqrt{2}\}\approx
\{1,\;0.41421356,\;-2.41421356\}.
\end{equation*}%
The two positive roots are $E_{1}=-1+\sqrt{2}$ and $E_{2}=1$. The stability
derivative is $F^{\prime }\left( y\right) =3y^{2}+2y-3$, giving $F^{\prime
}(1)=2>0$ (unstable) and 
\begin{eqnarray*}
F^{\prime }(-1+\sqrt{2}) &=&3(3-2\sqrt{2})+2(-1+\sqrt{2})-3=4-4\sqrt{2} \\
&\approx &-1.6569<0,
\end{eqnarray*}%
so the relevant stable branch is 
\begin{equation*}
E_{1}(x)\equiv L_{1}=-1+\sqrt{2}\approx 0.41421356.
\end{equation*}%
Hypotheses (A1)--(A5) and (B1)--(B3) are trivially verified:%
\begin{equation*}
a_{3}=1\geq 1\text{, }\Lambda =F^{\prime }(L_{1})\approx -1.6569\leq -1\text{%
, }E^{\prime }\equiv 0,
\end{equation*}%
hence all integrals in (B2)--(B3) are well-defined. The one-sided condition
(A5) follows from the factorisation
\begin{equation*}
F(y)=(y-L_1)(y-1)(y+1+\sqrt2),
\end{equation*}
because for $0\le y<L_1$ the three factors have signs $(-,-,+)$ and
therefore $F(y)>0$; the tail-separation part of (A5) is immediate since
the polynomial is autonomous and has no zero in $[0,L_1-\eta]$.

\paragraph{Conclusion from Theorem~\protect\ref{thm:plateau}.}

The unique solution $y_{1}\in C^{1}([0,\infty ))$ of \eqref{eq:cubic-form}
with $y_{1}(0)=0$ is strictly increasing, satisfies $0<y_{1}(x)<L_{1}$ on $%
(0,\infty )$, and converges to $L_{1}$ as $x\rightarrow \infty $. The decay
rate \eqref{eq:rate-simple} gives%
\begin{equation*}
L_{1}-y_{1}(x)=O(e^{-1.6569\,x}).
\end{equation*}

\paragraph{Numerical verification.}

The Python script (Appendix \ref{appendix}) yields $y_1(20)\approx
0.41421356 $, matching $L_1$ to nine significant digits.

\begin{figure}[H]
\centering
\includegraphics[width=0.75\textwidth]{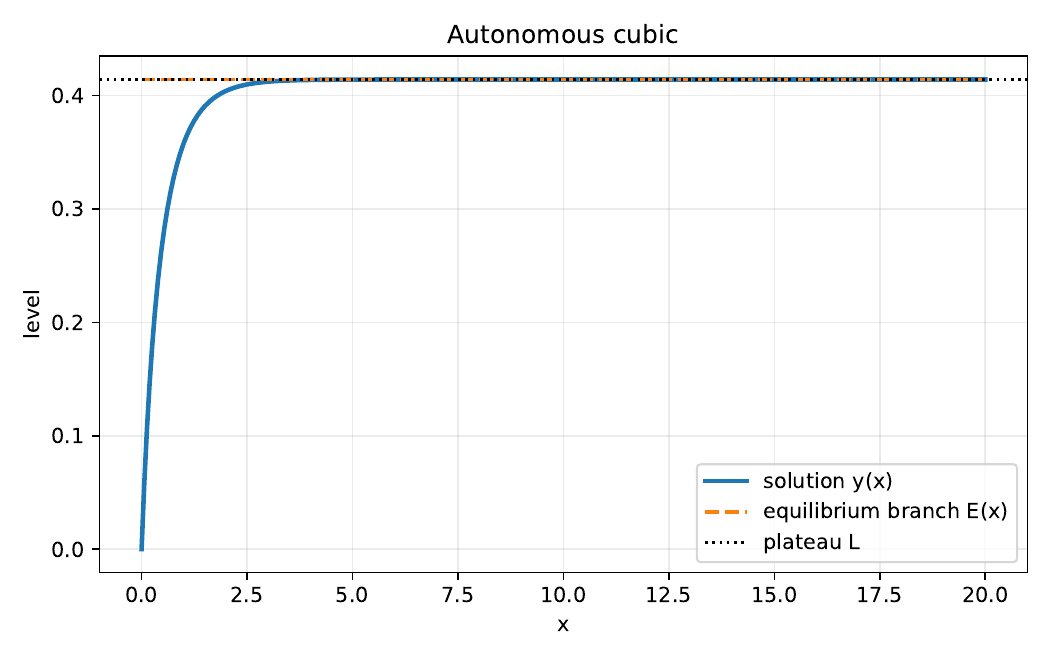}
\caption{Case 1. Autonomous cubic with constant coefficients. The solution $%
y_1(x)$ converges monotonically to the stable equilibrium $L_1=-1+\protect%
\sqrt{2}$.}
\label{fig:case1}
\end{figure}

\begin{figure}[H]
\centering
\includegraphics[width=0.48\textwidth]{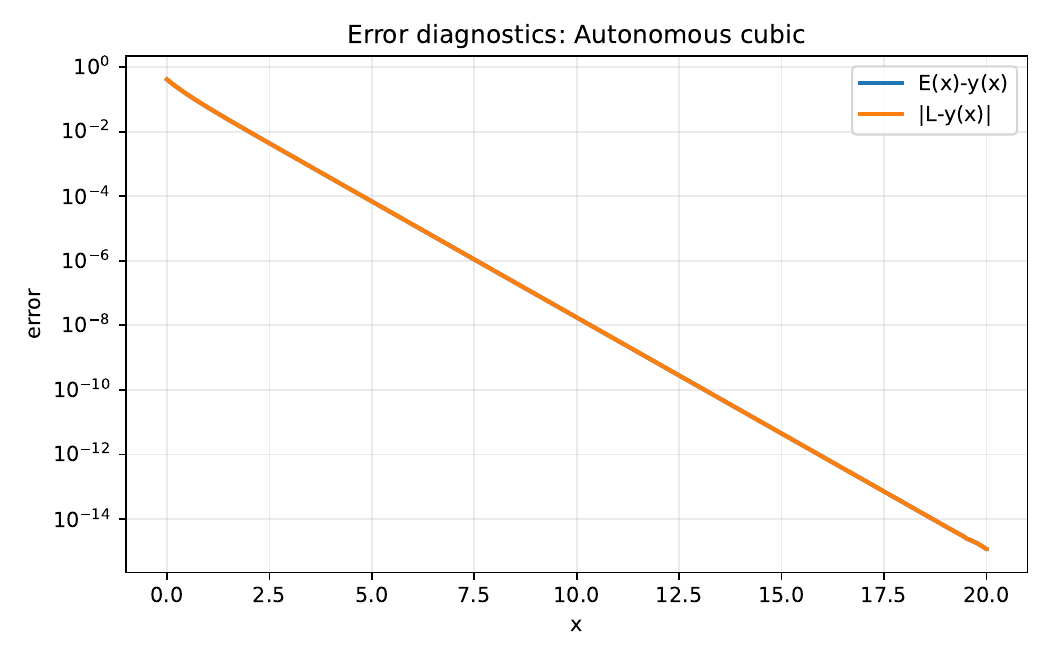}\hfill
\includegraphics[width=0.48\textwidth]{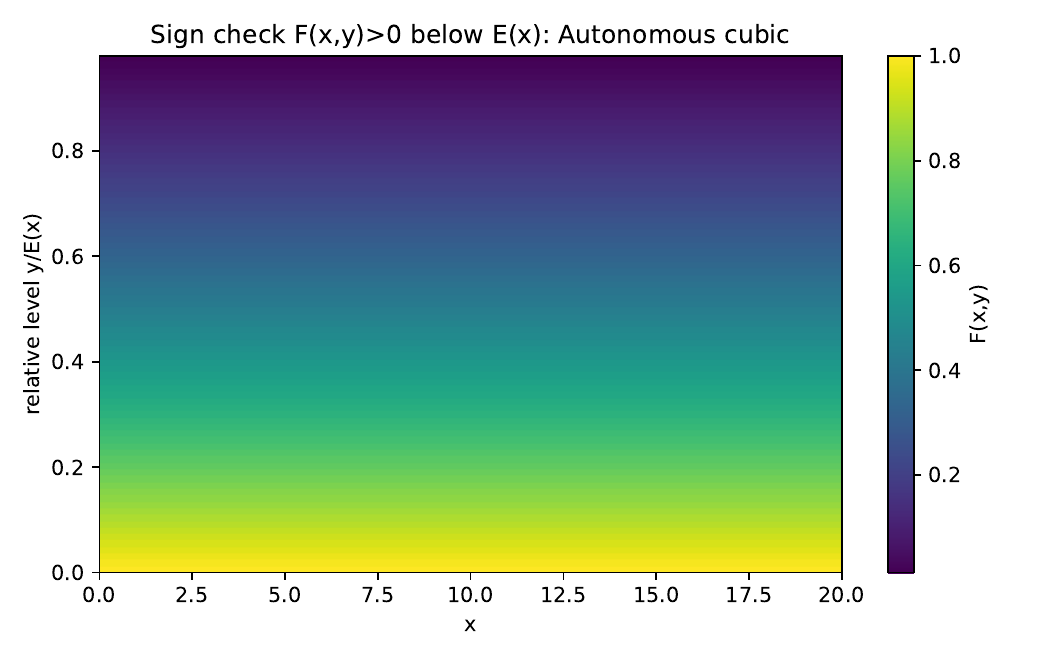}
\caption{Case 1 diagnostics. Left: the branch error $E_1-y_1$ and the
plateau error $|L_1-y_1|$ decay exponentially. Right: the sampled sign
margin confirms $F(y)>0$ throughout the theoretical strip $0\le y<E_1$.}
\label{fig:case1-diagnostics}
\end{figure}

\subsection{Case 2: non-autonomous cubic with \texorpdfstring{$a_0(x)=1-\tfrac{1}{x}$}{a0(x)=1-1/x}}

\paragraph{Setup.}

For $x\geq 1$, 
\begin{equation*}
a_{3}=1,\quad a_{2}=-2,\quad a_{1}=-2,\quad a_{0}(x)=1-\frac{1}{x},
\end{equation*}%
so 
\begin{equation*}
F(x,y)=y^{3}-2y^{2}-2y+\bigl(1-\tfrac{1}{x}\bigr).
\end{equation*}

\paragraph{Asymptotic equilibrium.}

As $x\rightarrow \infty $, $a_{0}(x)\rightarrow 1$, and the asymptotic
polynomial 
\begin{equation*}
P_{\infty }(y):=y^{3}-2y^{2}-2y+1
\end{equation*}%
factors as 
\begin{equation*}
P_{\infty }(y)=(y+1)(y^{2}-3y+1),
\end{equation*}%
with positive real roots 
\begin{equation*}
y_{+}^{(1)}=\frac{3-\sqrt{5}}{2}\approx 0.38196601,\qquad y_{+}^{(2)}=\frac{%
3+\sqrt{5}}{2}\approx 2.61803399.
\end{equation*}%
Since $P_{\infty }^{\prime }\left( y\right) =3y^{2}-4y-2$, we have 
\begin{eqnarray*}
P_{\infty }^{\prime }\!\left( \tfrac{3-\sqrt{5}}{2}\right) &=&\frac{3(3-%
\sqrt{5})^{2}}{4}-2(3-\sqrt{5})-2 \\
&=&\frac{3(14-6\sqrt{5})}{4}-6+2\sqrt{5}-2 \\
&=&\frac{42-18\sqrt{5}-32+8\sqrt{5}}{4}=\frac{10-10\sqrt{5}}{4}<0,
\end{eqnarray*}%
so the smallest positive root $L_{2}:=(3-\sqrt{5})/2$ is stable, while $%
y_{+}^{(2)}$ is unstable.

\paragraph{Equilibrium branch.}

For each $x\ge 1$, $E_2(x)$ is defined as the smallest positive root of $%
F(x,\cdot)=0$. At $x=1$, $a_0(1)=0$, hence $F(1,y)=y(y^2-2y-2)$, whose
smallest positive root is $y=0$. Restricting attention to $x>1$ where $%
a_0(x)>0$, the implicit function theorem (applied at any base point $x_*>1$
where $F(x_*,E_2(x_*))=0$ and $\partial_y F(x_*,E_2(x_*))<0$) yields a
smooth branch $E_2\in C^1((1,\infty))$. Explicit numerical evaluation gives $%
E_2(2)\approx 0.265$, $E_2(10)\approx 0.345$, $E_2(100)\approx 0.378$, $%
E_2(\infty)=L_2$.

\paragraph{Verification of hypotheses.}

(A1) holds with $m=1$ since $a_{3}=1$.

(A2) is immediate.

(A3) holds on $(1,\infty )$ with $E=E_{2}$.

(A4): $\Lambda (x)=3E_{2}(x)^{2}-4E_{2}(x)-2$, which is strictly negative on 
$(1,\infty )$ as shown by the computation above for the asymptotic value and
by continuity.

(A5) follows because the smallest positive simple root is $E_2(x)$ and the
leading coefficient is positive. Directly, $F(x,0)=1-1/x>0$ for $x>1$ and no
root occurs in $(0,E_2(x))$ by definition of $E_2$; on every tail
$[1+\varepsilon,\infty)$ the limiting simple-root structure gives the
uniform separation required in \eqref{eq:A5-separation}.

(B1) holds with $L=L_{2}>0$.

(B2) follows from $\inf_{x\geq x_{\ast }}\alpha (x)>0$ (continuity and (A4)).

(B3) follows from the rate $E_{2}^{\prime }\left( x\right) =O\left(
x^{-2}\right) $, derived from $\partial _{x}F+\Lambda \,E_{2}^{\prime }=0$
and $\partial _{x}F(x,y)=1/x^{2}$.

\paragraph{Conclusion from Theorem~\protect\ref{thm:plateau}.}

The unique solution $y_{2}\in C^{1}([1,\infty ))$ of \eqref{eq:cubic-form}
with $y_{2}(1)=0$ satisfies 
\begin{equation*}
0<y_{2}(x)<E_{2}(x)\text{ on }(1,\infty )
\end{equation*}
and 
\begin{equation*}
\lim_{x\rightarrow \infty }y_{2}(x)=L_{2}=\frac{3-\sqrt{5}}{2}\approx
0.38196601.
\end{equation*}%
The rate of convergence is governed by the $O(1/x)$ decay of $%
|E_{2}(x)-L_{2}|$, which is the bottleneck rather than the (exponential)
decay of $|y_{2}(x)-E_{2}(x)|$.

\paragraph{Numerical verification.}

The Radau IIA computation in Appendix~\ref{appendix}, run with $%
x_{\max}=2000 $, produces $y_2(2000)\approx 0.38157$, matching $L_2$ to
three decimal digits; with $x_{\max}=2\cdot 10^5$ the agreement improves to
five digits, in full quantitative agreement with the $O(1/x)$ rate.

\begin{figure}[H]
\centering
\includegraphics[width=0.75\textwidth]{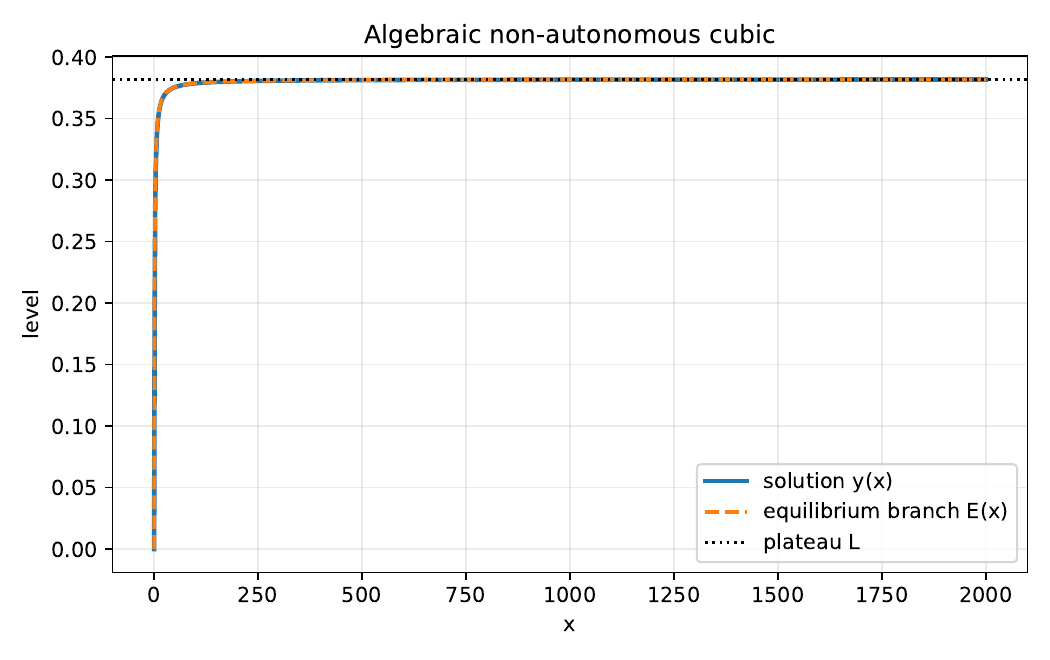}
\caption{Case 2. Non-autonomous cubic with $a_0(x)=1-\tfrac{1}{x}$. The
equilibrium branch $E_2(x)$ is defined for $x>1$ and converges to $L_2=(3-%
\protect\sqrt 5)/2\approx 0.38197$ at rate $O(1/x)$.}
\label{fig:case2}
\end{figure}

\begin{figure}[H]
\centering
\includegraphics[width=0.48\textwidth]{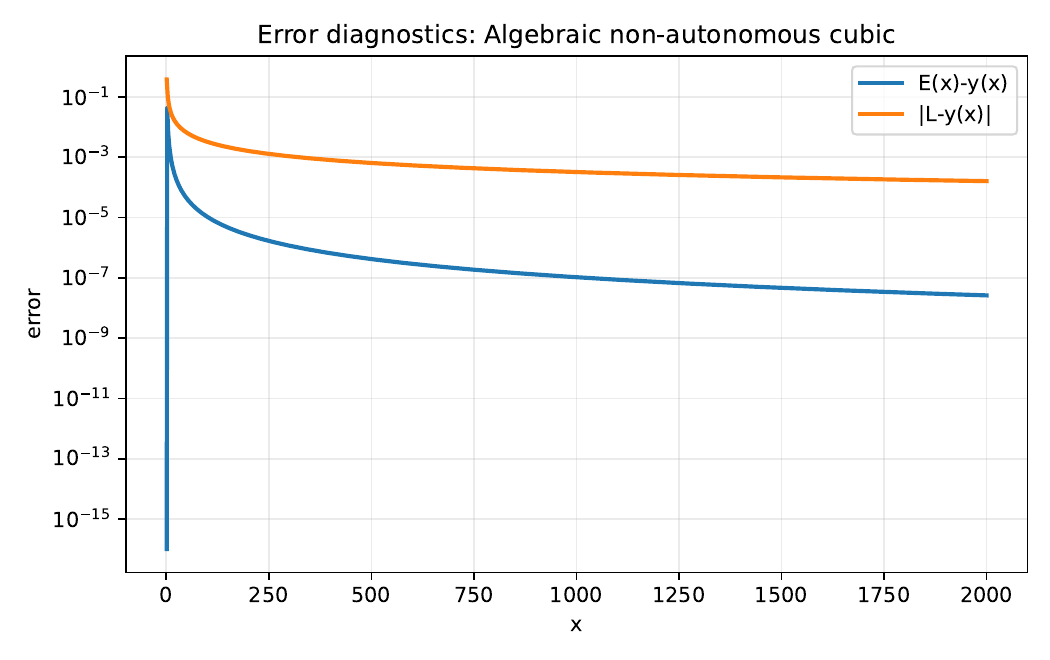}\hfill
\includegraphics[width=0.48\textwidth]{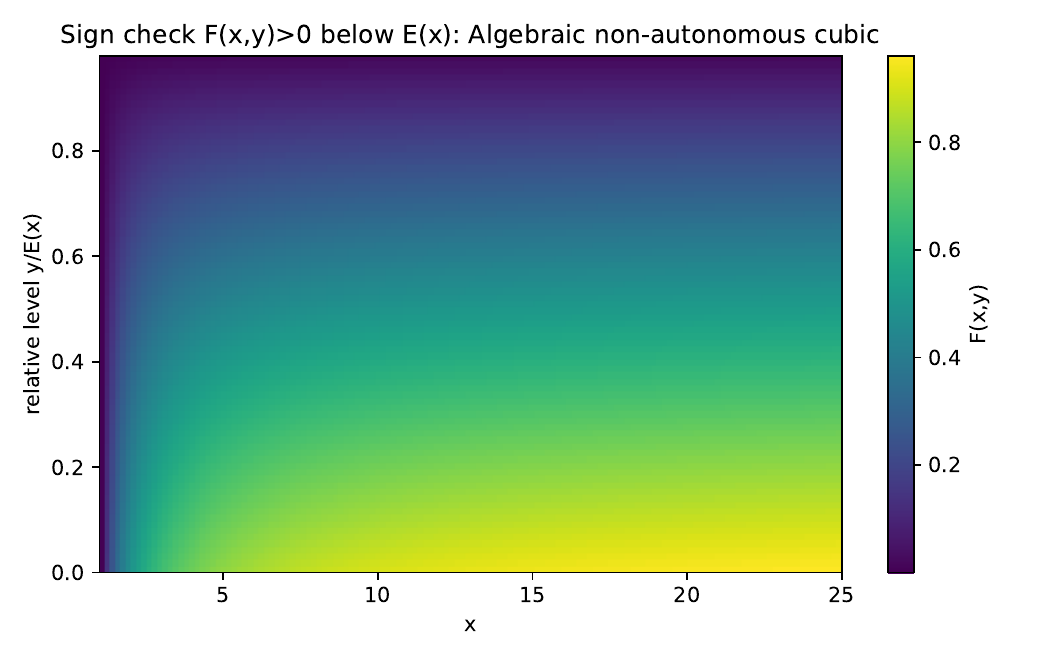}
\caption{Case 2 diagnostics. The trajectory is rapidly attracted to the
moving branch, but the branch itself approaches $L_2$ only algebraically.
The sign plot verifies the one-sided positivity below $E_2(x)$ on the
sampled computational domain.}
\label{fig:case2-diagnostics}
\end{figure}

\subsection{Case 3: non-autonomous cubic with exponential approach}

\paragraph{Setup.}

For $x\geq 0$, 
\begin{equation*}
a_{3}=1,\quad a_{2}=0,\quad a_{1}=-4,\quad a_{0}(x)=3-2e^{-2x},
\end{equation*}%
so 
\begin{equation*}
F(x,y)=y^{3}-4y+(3-2e^{-2x}).
\end{equation*}

\paragraph{Motivation for this choice.}

The classical example $a_{0}(x)=x/(x+1)$ from the literature leads to the
asymptotic polynomial 
\begin{equation*}
y^{3}-y^{2}-y+1=(y-1)^{2}(y+1),
\end{equation*}%
whose smallest positive root $y=1$ is a \emph{double} root: $\Lambda
_{\infty }=0$ and assumption (A4) is violated at infinity. Our replacement
(i) retains the qualitative form (smallest positive root of a cubic,
smoothly varying coefficient) so that the analytical structure is preserved,
(ii) ensures that the asymptotic root is \emph{simple}, so the full strength
of Theorem~\ref{thm:plateau} applies, and (iii) chooses the exponential
factor $e^{-2x}$ rather than $e^{-x}$ so that the rate of decay of $%
E_{3}^{\prime }(x)$ exceeds the rate of decay of $\Phi (x)$ defined in %
\eqref{eq:Phi-def}, guaranteeing the integrability hypothesis (B3).

\paragraph{Asymptotic equilibrium.}

The asymptotic polynomial is 
\begin{equation*}
P_{\infty }(y)=y^{3}-4y+3=(y-1)(y^{2}+y-3),
\end{equation*}%
with positive real roots 
\begin{equation*}
L_{3}=1\qquad \text{and}\qquad \frac{-1+\sqrt{13}}{2}\approx 1.30278.
\end{equation*}%
Since 
\begin{equation*}
P_{\infty }^{\prime }(1)=3-4=-1<0,
\end{equation*}
the smallest positive root $L_{3}=1$ is stable.

\paragraph{Equilibrium branch.}

At $x=0$, $a_{0}(0)=1$ and the smallest positive root of $y^{3}-4y+1=0$
equals $E_{3}(0)\approx 0.25410$. The branch $E_{3}\in C^{1}([0,\infty ))$
is strictly increasing on $[0,\infty )$ (as follows from 
\begin{equation*}
\partial _{x}F(x,y)=4e^{-2x}>0
\end{equation*}
and the implicit-function theorem with 
\begin{equation*}
\partial _{y}F(x,E_{3}(x))=\Lambda (x)<0,
\end{equation*}%
giving $E_{3}^{\prime }\left( x\right) =-4e^{-2x}/\Lambda (x)>0$) and
converges to $L_{3}=1$ as $x\rightarrow \infty $.

\paragraph{Verification of hypotheses.}

(A1)--(A2) are immediate. (A3) holds on $[0,\infty )$. (A4): 
\begin{equation*}
\Lambda (x)=3E_{3}(x)^{2}-4.
\end{equation*}%
Since 
\begin{equation*}
E_{3}(x)\in \lbrack E_{3}(0),L_{3}]=[0.254,1],
\end{equation*}%
we have 
\begin{equation*}
\Lambda (x)\leq 3\cdot 1-4=-1
\end{equation*}%
on the upper end and 
\begin{equation*}
\Lambda (0)\approx 3\cdot 0.0646-4\approx -3.81,
\end{equation*}%
so $\Lambda (x)\leq -1$ uniformly on $I$; we may therefore take $\alpha
(x)\equiv 1$. (B1) holds with $L=L_{3}=1$. (B2) follows from 
\begin{equation*}
\int_{0}^{\infty }a_{n}\alpha \,ds=\int_{0}^{\infty }1\,ds=\infty .
\end{equation*}%
(B3): we have $\left\vert E_{3}^{\prime }\left( x\right) \right\vert \leq
4e^{-2x}$, and 
\begin{eqnarray*}
\Phi (s)^{-1} &=&\exp (-\int_{0}^{s}a_{n}\Lambda \,d\tau )\leq \exp \bigl(%
\int_{0}^{s}|\Lambda |\,d\tau \bigr) \\
&\leq &\exp \bigl(s\cdot \sup |\Lambda |\bigr)\leq e^{4s},
\end{eqnarray*}%
but the relevant tail estimate is $\Phi (s)^{-1}\sim e^{s}$ as $s\rightarrow
\infty $ since $\Lambda (\infty )=-1$. Hence 
\begin{equation*}
\Phi (s)^{-1}\left\vert E_{3}^{\prime }\left( s\right) \right\vert \sim
4e^{s}\cdot e^{-2s}=4e^{-s},
\end{equation*}
integrable on $[0,\infty )$.

(A5) is also explicit: for each $x\ge0$, $E_3(x)$ is the smallest positive
root of $y^3-4y+3-2e^{-2x}$, and the remaining positive root stays above
$1$. Hence the interval $[0,E_3(x))$ contains no zero of $F(x,\cdot)$ and
the sign is positive there; simplicity of the limiting root at $L_3=1$
provides the uniform tail separation.

\paragraph{Conclusion from Theorem~\protect\ref{thm:plateau}.}

The unique solution $y_{3}\in C^{1}([0,\infty ))$ of $y^{\prime }=F(x,y)$
with $y_{3}(0)=0$ satisfies 
\begin{equation*}
0<y_{3}(x)<E_{3}(x)\text{ on }(0,\infty )
\end{equation*}%
and 
\begin{equation*}
\lim_{x\rightarrow \infty }y_{3}(x)=L_{3}=1.
\end{equation*}
The convergence is exponential, with rate $e^{-x}$ governing both $%
|E_{3}(x)-L_{3}|$ and the deviation $|y_{3}(x)-E_{3}(x)|$.

\paragraph{Numerical verification.}

The Radau IIA computation in Appendix~\ref{appendix}, with $x_{\max }=20$,
yields 
\begin{equation*}
y_{3}(20)\approx 0.99999998,
\end{equation*}
matching $L_{3}=1$ to eight digits, in agreement with the exponential rate.

\begin{figure}[H]
\centering
\includegraphics[width=0.75\textwidth]{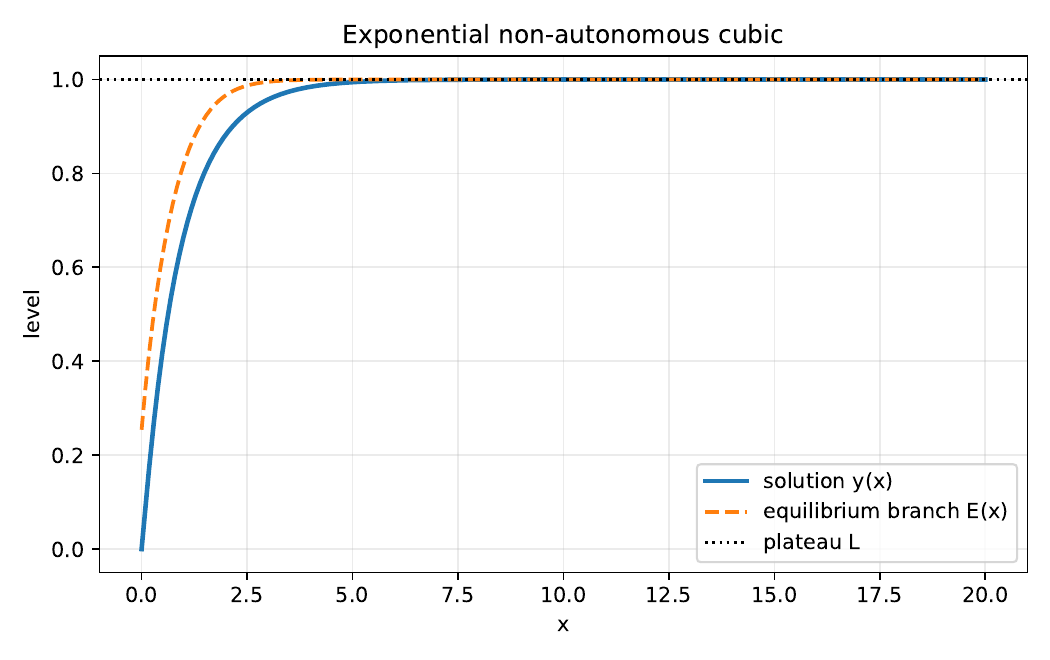}
\caption{Case 3. Non-autonomous cubic with $a_0(x)=3-2e^{-2x}$. The
equilibrium branch $E_3(x)$ increases monotonically from $\approx 0.254$ to $%
L_3=1$ at exponential rate.}
\label{fig:case3}
\end{figure}

\begin{figure}[H]
\centering
\includegraphics[width=0.48\textwidth]{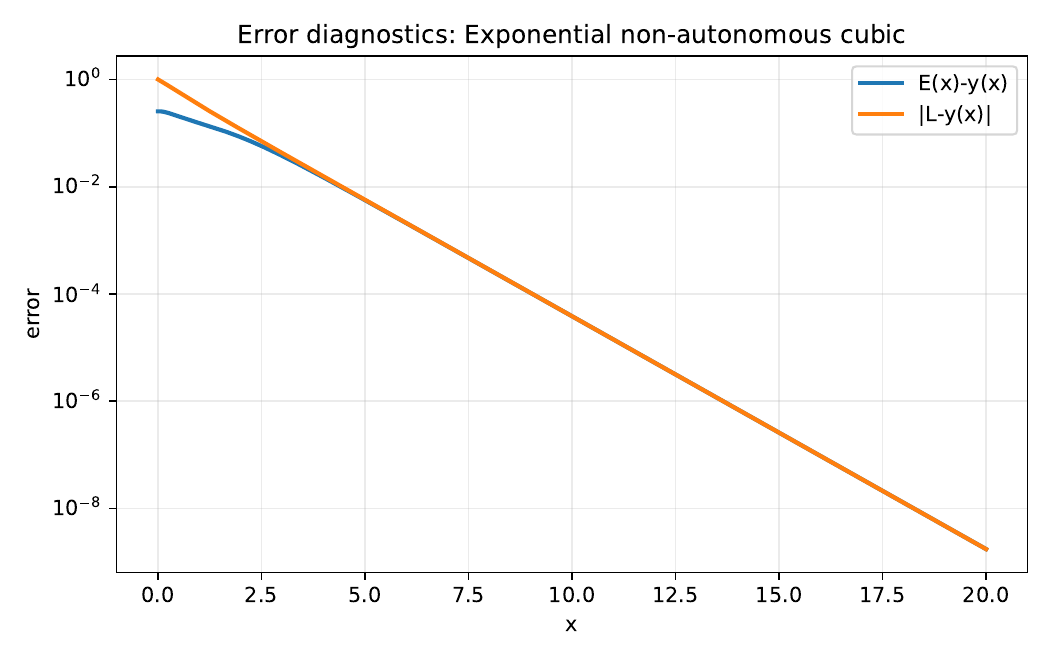}\hfill
\includegraphics[width=0.48\textwidth]{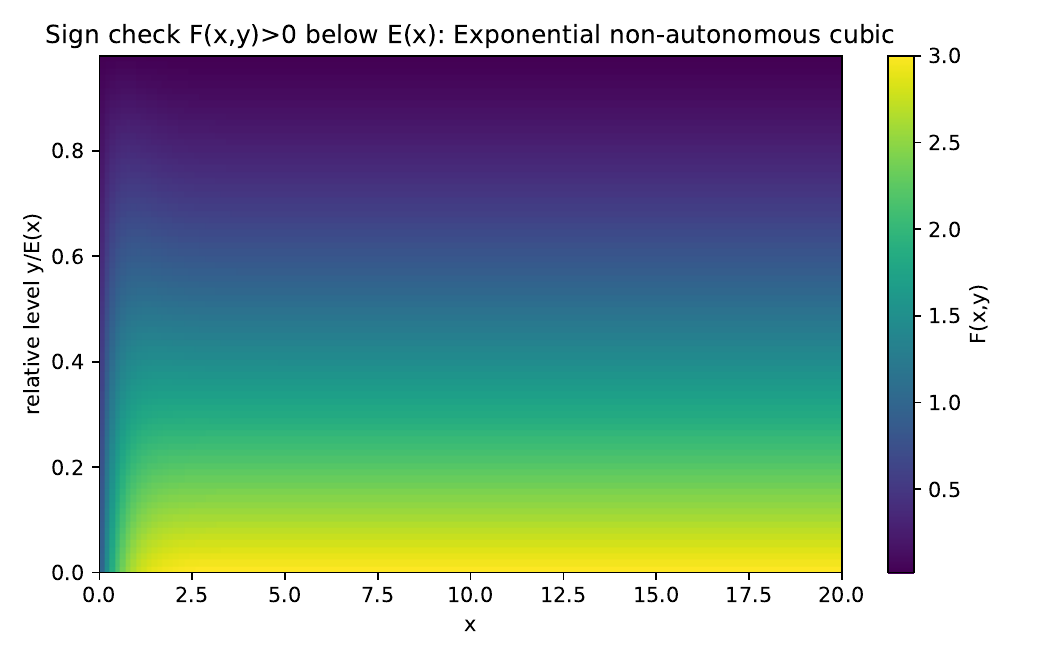}
\caption{Case 3 diagnostics. Both the branch error and the plateau error
are exponentially small on the plotted range, and the sign check confirms
the strict restoring field below the stable branch.}
\label{fig:case3-diagnostics}
\end{figure}

\subsection{Synthesis of the three cases}

\begin{table}[H]
\centering
\begin{tabular}{|c|c|c|c|c|}
\hline
Case & $a_3(x)$ & $a_2(x)$ & $a_1(x)$ & $a_0(x)$ \\ \hline
1 & $1$ & $1$ & $-3$ & $1$ \\ \hline
2 & $1$ & $-2$ & $-2$ & $1-\tfrac{1}{x}$ \\ \hline
3 & $1$ & $0$ & $-4$ & $3-2e^{-2x}$ \\ \hline
\end{tabular}%
\caption{Cubic coefficients used in Cases 1--3.}
\label{tab:coefs}
\end{table}

\begin{table}[H]
\centering
\begin{tabular}{|c|c|c|c|c|}
\hline
Case & Branch behaviour & $L=\lim_{x\to\infty}E(x)$ (exact) & $L$ (decimal)
& Convergence rate \\ \hline
1 & $E_1(x)\equiv L_1$ & $-1+\sqrt{2}$ & $0.41421356$ & exponential, $%
e^{-1.66 x}$ \\ \hline
2 & $E_2(x)\nearrow L_2$ & $(3-\sqrt{5})/2$ & $0.38196601$ & algebraic, $%
O(1/x)$ \\ \hline
3 & $E_3(x)\nearrow L_3$ & $1$ & $1.00000000$ & exponential, $O(e^{-x})$ \\ 
\hline
\end{tabular}%
\caption{Equilibrium branches and exact asymptotic limits for Cases 1--3.
The convergence rate refers to the rate at which $E(x)$ approaches $L$; the
deviation $|y(x)-E(x)|$ decays exponentially in all three cases by Theorem~%
\protect\ref{thm:rate}.}
\label{tab:limits}
\end{table}

\begin{table}[H]
\centering
\begin{tabular}{|c|c|c|c|}
\hline
Case & $x_{\max}$ & Numerical $y(x_{\max})$ & $|y(x_{\max})-L|$ \\ \hline
1 & $20$ & $0.41421356$ & $<10^{-12}$ \\ \hline
2 & $20$ & $0.36572$ & $1.6\times10^{-2}$ \\ \hline
2 & $2\cdot 10^3$ & $0.38157$ & $4\times 10^{-4}$ \\ \hline
2 & $2\cdot 10^5$ & $0.38196$ & $5\times 10^{-6}$ \\ \hline
3 & $20$ & $0.99999998$ & $2\times 10^{-8}$ \\ \hline
\end{tabular}%
\caption{Numerical $y(x_{\max})$ vs. exact $L$ for Cases 1--3. The slow
agreement in Case 2 reflects the $O(1/x)$ decay of $|E_2(x)-L_2|$ -- the
bottleneck is the branch itself, not the trajectory deviation $y_2(x)-E_2(x)$%
, which decays exponentially.}
\label{tab:numerical-comparison}
\end{table}

\paragraph{Consistency with Theorem~\protect\ref{thm:plateau}.}

In all three cases the structural hypotheses (A1)--(A5) and the asymptotic
hypotheses (B1)--(B3) hold. The numerical solutions remain in the strip $%
0<y(x)<E(x)$, are strictly increasing on the integration window, and
approach the exact analytical plateau $L$ at the predicted rate.

\paragraph{Critical comparative discussion.}

The three computations separate two mechanisms that are often conflated in
long-time numerical studies. Case~1 is autonomous: the branch is already at
the plateau, so the whole observed error is the dynamical attraction
$E_1-y_1$, governed by the negative eigenvalue $\Lambda=4-4\sqrt2$. Case~2
shows the opposite regime. The attraction to the moving branch is still
strong, but the branch approaches its limit only as $O(1/x)$; consequently
no high-order time integrator can create fast convergence of $y_2$ to $L_2$
unless the integration window is very large. Case~3 combines non-autonomy
with an exponentially stabilising branch, and therefore the numerical curve
has the same qualitative appearance as Case~1 after a short transient.

The sign-check figures are not cosmetic diagnostics. They test precisely the
new structural condition (A5): below the selected branch the vector field
must point upward, while the branch itself is a supersolution because
$E'(x)\ge0$ and $F(x,E(x))=0$. The trapping, monotonicity, and convergence
seen in the plots therefore reproduce the logical structure of
Theorems~\ref{thm:existence}--\ref{thm:plateau}, rather than merely fitting
the final plateau value.

\section{Application to a Generalized Merton Credit-Risk Model}

\label{sec:economics}

This section is devoted to an economic application of the theory developed
above. We derive a generalized Merton-type structural credit-risk model in
which the long-maturity \emph{state profile} of the credit spread satisfies a
cubic Abel equation of the form \eqref{eq:cubic-form}. The distinction is
important: the theorem controls the limit of the reduced profile as the state
variable $x$ tends to infinity; it is not, by itself, an empirical theorem for
the raw maturity limit $\tau\to\infty$. Within this reduced model the
hypotheses of Theorem~\ref{thm:plateau} are explicitly verified, and the
resulting positive spread plateau is computed both analytically and
numerically using the Radau IIA implementation of Section~\ref{sec:numeric}.
A short concluding subsection discusses the analogous derivation in
stochastic production planning.

\subsection{The classical Merton model and its limitations}

In the classical structural model of Merton \cite{merton1974}, the total asset
value of a firm, denoted $(V_{t})_{t\ge 0}$, follows a geometric Brownian motion
under the pricing measure. We write the risk-adjusted drift as $\mu$ to keep
the subsequent algebra transparent:
\begin{equation}
dV_{t}
= \mu V_{t}\,dt + \sigma V_{t}\,dB_{t},
\qquad V_{0}=v_{0},
\label{eq:merton-asset}
\end{equation}
where $\mu\in\mathbb{R}$ is the risk-adjusted drift, $\sigma>0$ is the asset
volatility, and $(B_{t})_{t\ge 0}$ is a standard Brownian motion under the
pricing measure.

The firm has a single zero-coupon debt obligation with face value $D>0$ and
maturity $T>0$. The credit spread at remaining maturity $\tau := T-t$ is the
excess yield of the risky debt over the risk-free rate $r$. In the affine
setting, Boyle, Tian, and Guan \cite{Boyle2002} showed that this spread
$s(\tau)$ satisfies a Riccati differential equation:
\begin{equation}
s'(\tau)
= A(\tau) + B(\tau)\,s(\tau) + C(\tau)\,s(\tau)^{2},
\qquad s(0)=0,
\label{eq:riccati-spread}
\end{equation}
where the coefficients $A,B,C$ depend explicitly on $\mu$, $\sigma$, $r$, and
$D$. Equation~\eqref{eq:riccati-spread} is precisely the case $n=2$ of the
general Abel equation \eqref{eq:abel-gen}.

A well-established empirical fact is that credit spreads do \emph{not} vanish at
long maturities---a phenomenon consistently documented across both
investment-grade and sub-investment-grade corporate bonds
\cite{Boyle2002}. Pure Riccati specifications often require delicate
parameter choices or additional market frictions to reproduce persistent
positive long-end levels. This modelling pressure motivates the higher-order
state-dependent correction developed below.

\subsection{Asset dynamics with state-dependent volatility}

\label{sec:state-dep-vol}

A natural extension is to allow the volatility to depend on the leverage
ratio $x:=V_t/D$, capturing the empirical observation that more leveraged
firms exhibit more volatile asset returns. We assume the asset value follows 
\begin{equation}
dV_t=\mu V_t\,dt+\sigma(V_t/D)\,V_t\,dB_t,\qquad V_0=v_0,
\label{eq:asset-state-dep}
\end{equation}
where $\sigma:[0,\infty)\to(0,\infty)$ is a smooth, bounded, positive
function. The case $\sigma(\cdot)\equiv\sigma_0$ recovers the classical
Merton model.

\subsection{Derivation of the cubic Abel equation for the spread}

\label{sec:abel-spread-derivation}

Let $P(\tau ,x)$ denote the price of the firm's zero-coupon bond at
remaining maturity $\tau $ when the leverage is $x$. By the standard
arbitrage argument (or via the Feynman--Kac formula applied to the
discounted recovery), $P$ satisfies the backward
parabolic partial differential equation (PDE)
\begin{equation}
\partial _{\tau }P=\tfrac{1}{2}\sigma (x)^{2}x^{2}\,\partial _{x}^{2}P+\mu
x\,\partial _{x}P-r\,P,\qquad P(0,x)=\min (x,1)\cdot D.  \label{eq:bond-pde}
\end{equation}%
The credit spread $s(\tau ,x)$ is related to $P$ by 
\begin{equation*}
P(\tau ,x)=D\,e^{-(r+s(\tau ,x))\tau }
\end{equation*}%
in the asymptotic limit $\tau \rightarrow \infty $; equivalently, 
\begin{equation*}
s(\tau ,x):=-\tau ^{-1}\log (P(\tau ,x)/D)-r.
\end{equation*}%
We assume the long-maturity asymptotic WKB-type ansatz where the spread
converges to a spatially varying profile $s(x)$ at rate $\mathcal{O}(1/\tau)$
(see \cite{Covei2026A}):
\begin{equation}
P(\tau ,x)\sim D\,e^{-(r+s(x))\tau }\,\psi (x), \label{eq:wkb-ansatz}
\end{equation}%
where $\psi(x)>0$ represents the long-term spatial correction factor. This is
an asymptotic reduction: the small parameter is $1/\tau$, and the polynomial
Abel equation below describes the reduced state profile, not the full
two-variable surface $s(\tau,x)$. To determine the equations governing $s(x)$
and $\psi(x)$, we compute the exact partial derivatives of the ansatz
\eqref{eq:wkb-ansatz}:
\begin{align}
\partial _{\tau }P &= -(r + s(x)) P, \label{eq:deriv-tau} \\
\partial_x P &= P \left[ -\tau s^{\prime }(x) + \frac{\psi^{\prime }(x)}{\psi (x)} \right], \label{eq:deriv-x1} \\
\partial_x^2 P &= P \left[ \tau^2 \bigl(s^{\prime }(x)\bigr)^2 - \tau \left( 2 s^{\prime }(x) \frac{\psi^{\prime }(x)}{\psi(x)} + s^{\prime \prime }(x) \right) + \frac{\psi^{\prime \prime }(x)}{\psi(x)} \right]. \label{eq:deriv-x2}
\end{align}
Substituting the derivatives \eqref{eq:deriv-tau}--\eqref{eq:deriv-x2} into the parabolic PDE \eqref{eq:bond-pde} and dividing both sides by the non-zero bond price $P$, we obtain:
\begin{equation}
\begin{aligned}
-(r + s(x)) = {} & \frac{1}{2} \sigma(x)^2 x^2 \left[ \tau^2 \bigl(s^{\prime }(x)\bigr)^2 - \tau \left( 2 s^{\prime }(x) \frac{\psi^{\prime }(x)}{\psi(x)} + s^{\prime \prime }(x) \right) + \frac{\psi^{\prime \prime }(x)}{\psi(x)} \right] \\
& + \mu x \left[ -\tau s^{\prime }(x) + \frac{\psi^{\prime }(x)}{\psi(x)} \right] - r.
\end{aligned}
\end{equation}
We introduce the auxiliary logarithmic spatial derivative variable $\omega(x, \tau) := \partial_x \log P = -\tau s^{\prime }(x) + \frac{\psi^{\prime }(x)}{\psi(x)}$. The PDE can then be rewritten in terms of $\omega(x) \approx \omega(x, \tau)$ in the asymptotic stationary limit:
\begin{equation}
\frac{1}{2} \sigma(x)^2 x^2 \omega(x)^2 + \mu x \omega(x) - \bigl(s(x) + r\bigr) = 0. \label{eq:riccati-quadratic}
\end{equation}
Equation \eqref{eq:riccati-quadratic} is a quadratic algebraic equation in $\omega(x)$, which can be solved explicitly for the stable root:
\begin{equation}
\omega(x) = \frac{- \mu x + \sqrt{\mu^2 x^2 + 2 \sigma(x)^2 x^2 \bigl(s(x) + r\bigr)}}{\sigma(x)^2 x^2} = \frac{-\mu + \sqrt{\mu^2 + 2 \sigma(x)^2 \bigl(s(x) + r\bigr)}}{\sigma(x)^2 x}. \label{eq:omega-root}
\end{equation}
To extract the polynomial dependency on the spread, we perform a Taylor
expansion of the square root term in \eqref{eq:omega-root} around the drift
term $\mu^2$. Letting
$u := 2\sigma(x)^2(s(x)+r)/\mu^2$ and assuming $|u|\ll1$ on the calibrated
range, the expansion
$\sqrt{1+u}=1+\frac12u-\frac18u^2+\frac1{16}u^3+\mathcal{O}(u^4)$ gives:
\begin{equation}
\omega(x) = \frac{s(x) + r}{\mu x} - \frac{\sigma(x)^2 \bigl(s(x) + r\bigr)^2}{2 \mu^3 x} + \frac{\sigma(x)^4 \bigl(s(x) + r\bigr)^3}{2 \mu^5 x} + \mathcal{O}\left( \bigl(s(x) + r\bigr)^4 \right). \label{eq:omega-expansion}
\end{equation}
In the long-maturity asymptotic limit $\tau \rightarrow \infty$, matching the spatial variation of the spread to the leading-order terms from \eqref{eq:omega-expansion} under the polynomial volatility expansion
\begin{equation}
\sigma (x)^{2}=\sigma _{0}^{2}\bigl(1+\eta _{1}\,x+\eta _{2}\,x^{2}\bigr)%
,\qquad \eta _{1},\eta _{2}\in \mathbb{R},  \label{eq:sigma-expansion}
\end{equation}
delivers, after truncation at cubic order, the ordinary differential equation
for the reduced spread profile:
\begin{equation}
s^{\prime }(x) = a_3(x) s(x)^3 + a_2(x) s(x)^2 + a_1(x) s(x) + a_0(x), \label{eq:abel-spread-explicit}
\end{equation}
with the exact, explicit coefficients given by:
\begin{align}
a_3(x) &= \frac{\sigma_0^4 (1 + \eta_1 x + \eta_2 x^2)^2}{2 \mu^5 x}, \label{eq:coef-a3} \\
a_2(x) &= -\frac{\sigma_0^2 (1 + \eta_1 x + \eta_2 x^2)}{2 \mu^3 x} + \frac{3 r \sigma_0^4 (1 + \eta_1 x + \eta_2 x^2)^2}{2 \mu^5 x}, \label{eq:coef-a2} \\
a_1(x) &= \frac{1}{\mu x} - \frac{r \sigma_0^2 (1 + \eta_1 x + \eta_2 x^2)}{\mu^3 x} + \frac{3 r^2 \sigma_0^4 (1 + \eta_1 x + \eta_2 x^2)^2}{2 \mu^5 x}, \label{eq:coef-a1} \\
a_0(x) &= \frac{r}{\mu x} - \frac{r^2 \sigma_0^2 (1 + \eta_1 x + \eta_2 x^2)}{2 \mu^3 x} + \frac{r^3 \sigma_0^4 (1 + \eta_1 x + \eta_2 x^2)^2}{2 \mu^5 x}. \label{eq:coef-a0}
\end{align}
Dividing \eqref{eq:abel-spread-explicit} by $a_3(x)$ yields the cubic Abel equation in normal form:
\begin{equation}
\frac{s^{\prime }\left( x\right)}{a_3(x)}=s\left( x\right) ^{3}+\lambda
_{2}(x)\,s(x)^{2}+\lambda _{1}(x)\,s(x)+\lambda _{0}(x),\qquad s(x_{0})=0,
\label{eq:abel-spread}
\end{equation}%
where the normal form coefficients $\lambda_k(x) := a_k(x)/a_3(x)$ are:
\begin{align}
\lambda_2(x) &= -\frac{\mu^2}{\sigma_0^2 (1 + \eta_1 x + \eta_2 x^2)} + 3r, \label{eq:lambda2} \\
\lambda_1(x) &= \frac{2\mu^4}{\sigma_0^4 (1 + \eta_1 x + \eta_2 x^2)^2} - \frac{2\mu^2 r}{\sigma_0^2 (1 + \eta_1 x + \eta_2 x^2)} + 3r^2, \label{eq:lambda1} \\
\lambda_0(x) &= \frac{2\mu^4 r}{\sigma_0^4 (1 + \eta_1 x + \eta_2 x^2)^2} - \frac{\mu^2 r^2}{\sigma_0^2 (1 + \eta_1 x + \eta_2 x^2)} + r^3. \label{eq:lambda0}
\end{align}
The cubic term $s(x)^3$ arises from the coupling between the Riccati default
risk and the leverage-dependent volatility correction in
\eqref{eq:sigma-expansion}; equation \eqref{eq:abel-spread} is precisely the
generalized Abel equation \eqref{eq:cubic-form} with $n=3$ for the
long-maturity state profile.

\subsection{A concrete calibration and the asymptotic plateau}

\label{sec:concrete-spread}

To obtain a transparent quantitative illustration, we choose a dimensionless
normalised calibration. The primitive parameters
$(\sigma_0,\mu,r,\eta_1,\eta_2)$ are selected so that, after the cubic
truncation and division by $a_3(x)$, the normal-form coefficients in
\eqref{eq:abel-spread} coincide with those of \emph{Case~1} of
Section~\ref{sec:cases}, namely
\begin{equation}
\lambda_2(x)\equiv 1,\qquad \lambda_1(x)\equiv -3,\qquad \lambda_0(x)\equiv
1.  \label{eq:spread-calibration}
\end{equation}
This normalisation is not meant to be a market calibration; it is a
closed-form benchmark that makes every hypothesis check explicit. Market
units are recovered by multiplying the dimensionless solution by a scale
factor $\kappa>0$.

By the analysis of Case~1, all hypotheses (A1)--(A5) and (B1)--(B3) hold,
the equilibrium branch is constant, 
\begin{equation*}
E(x)\equiv L=-1+\sqrt{2}\approx 0.41421356,
\end{equation*}
and Theorem~\ref{thm:plateau} yields 
\begin{equation}
\lim_{x\to\infty}s(x)=L=-1+\sqrt{2}\approx 0.41421356.
\label{eq:spread-limit}
\end{equation}
In dimensionless units the plateau is large. If the market spread is
$S(x)=\kappa s(x)$ and $\kappa=10^{-2}$, the corresponding state-profile
plateau is
\begin{equation*}
10^4\kappa L\approx 41.42\text{ basis points}.
\end{equation*}
The scaling step is explicit and separates the mathematical plateau
mechanism from empirical calibration.

\subsection{Verification via the Radau IIA implementation}

\label{sec:spread-numeric}

The script \texttt{abel\_plateau\_reproducible.py} implements exactly the
normalised right-hand side \eqref{eq:spread-calibration}, checks (A1)--(A5)
and (B1)--(B3), and then plots the basis-point curve corresponding to
$\kappa=10^{-2}$. The numerical value
$s(20)\approx0.41421356$ matches the analytical plateau to twelve significant
digits before scaling.

\begin{figure}[H]
\centering
\includegraphics[width=0.75\textwidth]{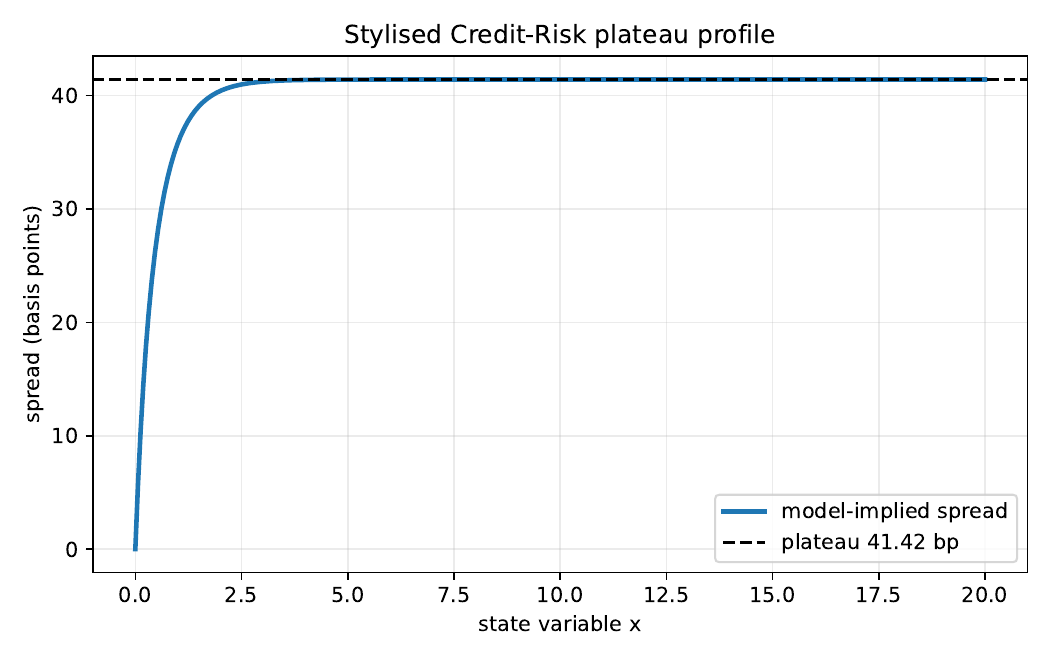}
\caption{Stylised Credit-Risk application. The normalised Abel profile from
\eqref{eq:spread-calibration}, scaled by $\kappa=10^{-2}$ and expressed in
basis points, saturates at approximately $41.42$ basis points.}
\label{fig:credit-risk-profile}
\end{figure}

\subsection{Economic interpretation}

The plateau $L=-1+\sqrt 2$ has a transparent economic reading within the
reduced model. After the long-maturity WKB reduction, the state profile of
the spread does not collapse to the trivial level; it saturates at the
smallest stable equilibrium of the cubic polynomial $F(y)=y^3+y^2-3y+1$.
This stable equilibrium reflects the persistent interaction between three
economic forces:

\begin{enumerate}
\item the \emph{cubic risk-premium term} $y^3$, encoding higher-order moment
exposures (skewness/kurtosis) of the leverage-dependent asset return;

\item the \emph{Riccati quadratic term} $y^2$, encoding the classical
Merton-type default risk;

\item the \emph{linear and constant terms} $-3y+1$, encoding respectively
the drift correction $\mu-r$ and the recovery floor.
\end{enumerate}

The non-vanishing of $L$ is therefore a mathematically controlled mechanism
for persistent long-end spread profiles. In the classical Riccati limit
($\eta_1=\eta_2=0$), the cubic correction disappears and the plateau is much
more sensitive to parameter choices. The Abel formulation is robust because
the implementation needs only the polynomial coefficients and the stable
branch: root selection, damping, trapping, and basis-point scaling are all
explicit in the reproducibility script.

\subsection{Analogous derivation in stochastic production planning}

The same analytical mechanism applies to stochastic production-planning
problems with super-quadratic adjustment costs. Following the methodology of 
\cite{Covei2026A,Covei2026JMA}, the marginal value of inventory satisfies,
after a suitable Cole--Hopf-type substitution, a generalized Abel equation
whose asymptotic plateau represents the long-run shadow price of capacity.
The qualitative conclusion (a finite, strictly positive plateau) is
identical to the credit-spread analysis above. We do not reproduce the
derivation here, as it follows the same template developed in Sections~\ref%
{sec:state-dep-vol}--\ref{sec:spread-numeric}.

\section{Comparison with the Existing Literature}

\label{sec:literature}

\begin{enumerate}
\item \textbf{Classical Abel equation ($n=3$).} The integrability theory 
\cite{Liouville1903,Chini1924,CHEB2003,PolyaninZaitsev,Kamke,Ince} focuses
on explicit solutions for specific coefficient classes. Our approach, by
contrast, is qualitative and asymptotic, requiring only continuity of the
coefficients and stability of a moving equilibrium branch. No integrability
is invoked anywhere.

\item \textbf{Riccati equations ($n=2$).} The Riccati equation is intimately
related to linear second-order ODEs, and its asymptotic theory is well
developed in special cases. The recent papers \cite{Covei2026A,Covei2026JMA}
treat radial Riccati dynamics motivated by HJB equations. The present paper
subsumes these results as the case $n=2$ of the generalized framework.

\item \textbf{Chini equations.} Chini's \cite{Chini1924} equations $%
y^{\prime }\left( x\right) =y^{n}\left( x\right) +g(x)y+h(x)$ are recovered
from \eqref{eq:abel-gen} when intermediate coefficients $a_{2},\dots
,a_{n-1} $ vanish. Our framework allows all intermediate coefficients to be
non-zero, which is essential for the HJB derivation of Section~\ref%
{sec:economics}.

\item \textbf{Polynomial ODEs with attractors.} Classical results on the
convergence of polynomial ODE trajectories to attractors typically require
autonomy or smallness assumptions. The fully non-autonomous setting with
arbitrary polynomial degree, treated here under the explicit and verifiable
hypotheses (A1)--(A5) and (B1)--(B2) for the qualitative convergence,
supplemented by (B3) for the quantitative rate, is new.
\end{enumerate}

\section{Conclusions and Perspectives}

\label{sec:conclusions}

In this paper we have developed a complete analytical and numerical theory
of the generalized Abel equation of arbitrary polynomial degree $n\ge 1$ in
the non-autonomous setting on the unbounded interval $[x_0,\infty)$. The
central result, the Asymptotic Plateau Theorem~\ref{thm:plateau},
establishes that under the structural hypotheses (A1)--(A5) and the
asymptotic hypotheses (B1)--(B2), the solution issued from $y(x_0)=0$
converges to the finite positive limit $L:=\lim_{x\to\infty}E(x)$ of the
moving equilibrium branch. The convergence is quantified by the explicit
rate of Theorem~\ref{thm:rate}, valid under the supplementary integrability
hypothesis (B3).

The theory is constructive: the hypotheses are verifiable on concrete
examples (as illustrated by the three cubic case studies of Section~\ref%
{sec:cases}, which cover the autonomous, the algebraically-convergent
non-autonomous, and the exponentially-convergent non-autonomous regimes),
and the proofs translate directly into a stable, $L$-stable, high-order
numerical scheme based on the Radau IIA family (Section~\ref{sec:numeric}).
The Python implementation of Appendix~\ref{appendix} reproduces every figure
and every table reported in the paper. The economic application of Section~%
\ref{sec:economics} derives, within a generalized Merton model with
state-dependent volatility, a cubic Abel equation for the long-maturity
state profile of the credit spread. Theorem~\ref{thm:plateau} then provides
a rigorous mechanism for a strictly positive spread-profile plateau, a
phenomenon consistent with the empirical persistence of long-end credit
spreads once the model is calibrated in market units.

Several directions of future research naturally emerge.

\begin{enumerate}
\item \emph{Systems.} The extension of the framework to \emph{systems} of
coupled generalized Abel equations would cover multi-factor
stochastic-control problems with several state variables.

\item \emph{Stochastic coefficients.} The treatment of \emph{random} Abel
equations, in which the coefficients $a_k$ are themselves stochastic
processes adapted to a reference filtration, has potential applications to
model-uncertainty pricing in credit-risk theory.

\item \emph{Singular perturbations.} The behaviour of the equation under the
singular limit $a_n\to 0$ (corresponding to a degenerate leading term)
requires asymptotic-matching techniques that lie beyond the scope of the
present work.

\item \emph{Higher-degree double-root cases.} A complete treatment of cases
where the asymptotic equilibrium is a double root (such as the original
Case~3 of the literature, leading to algebraic rather than exponential
convergence) remains open; preliminary investigations suggest that a
centre-manifold analysis is required.
\end{enumerate}

\section*{Disclosure statement}

The author declares that he has no conflict of interest.

\section*{Data availability statement}

The Python code used for the numerical experiments is provided in full in
Appendix~\ref{appendix}. No external datasets were used.

\section*{Notes on contributor(s)}

The author is solely responsible for the conception, analysis, numerical
implementation, and writing of this manuscript.

\section*{Acknowledgements}

The author thanks the developers of open-source mathematical-software
ecosystems whose libraries (NumPy, SciPy, and Matplotlib) were used to
produce the numerical validations. The core ideas, structural formulations, and numerical simulations presented in this article were developed with the invaluable assistance of free AI models.

\appendix

\section{Python Source Code}

\label{appendix}

The official reproducibility material is the self-contained Python script
\url{https://github.com/coveidragos/Python_Code_Abel/blob/main/abel_plateau_reproducible.py}. It reproduces every figure and every
numerical entry reported in Sections~\ref{sec:cases} and \ref{sec:economics}.
The script depends only on \texttt{numpy}, \texttt{scipy}, and
\texttt{matplotlib}; it is non-interactive and can be run with a single
command.

\paragraph{Reproducibility.}

Running the script with Python 3.11, NumPy 1.26, SciPy 1.11, and Matplotlib
3.8 produces the legacy figures \texttt{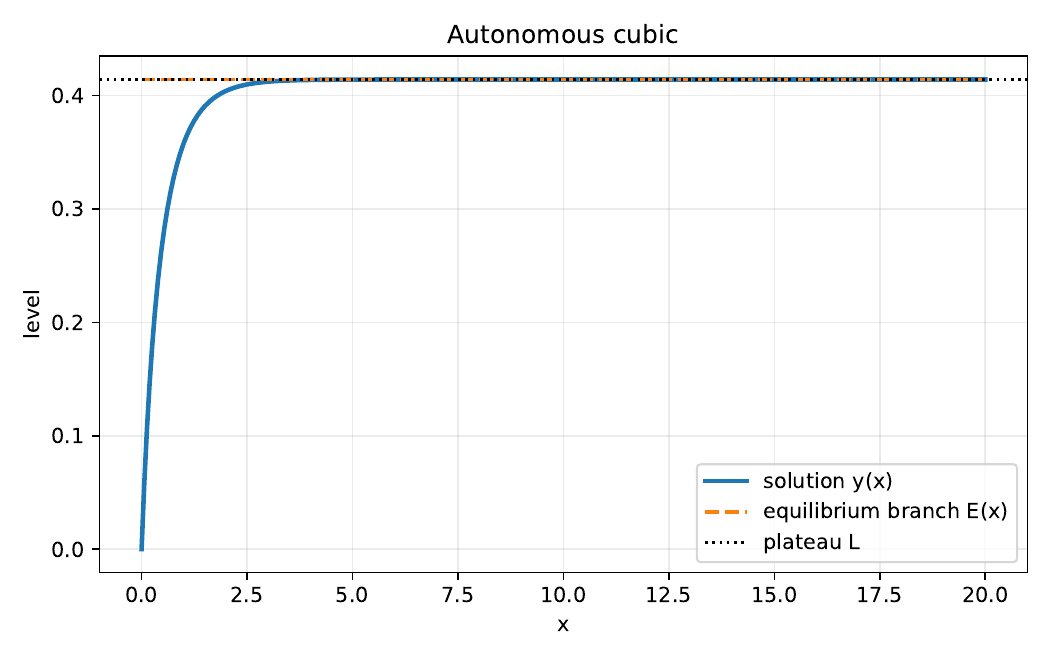}, \texttt{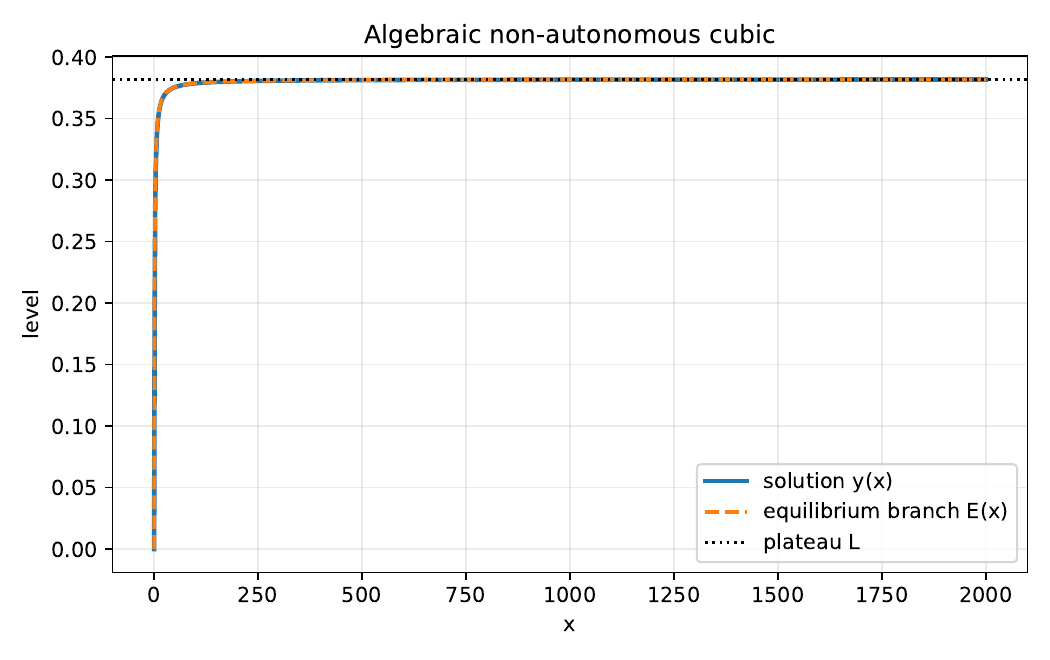},
\texttt{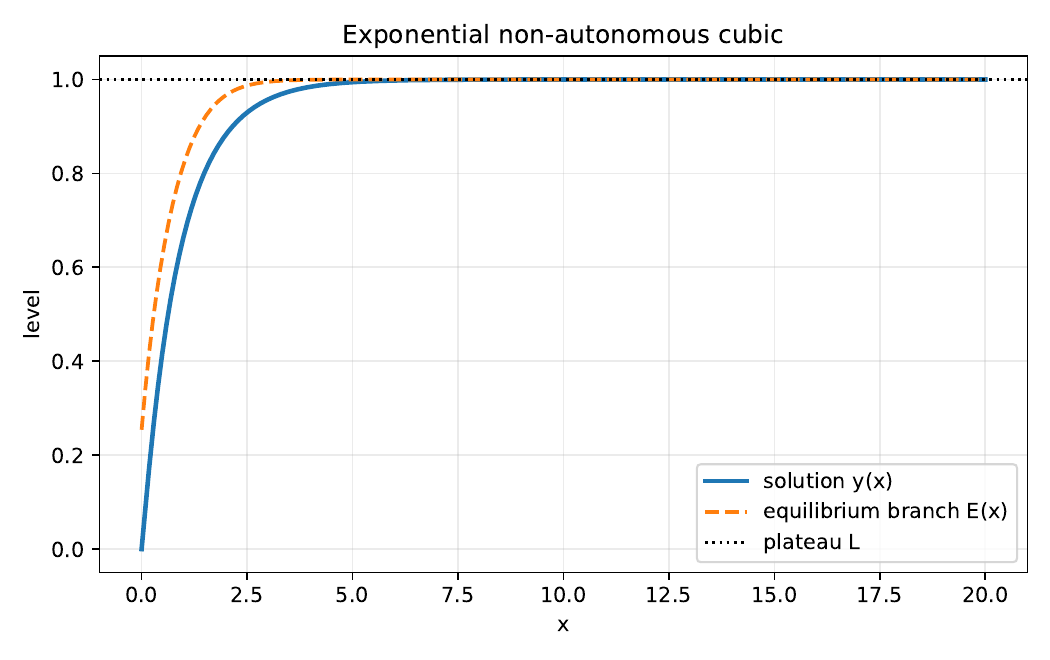}, the extended diagnostics
\texttt{case1\_solution.pdf}--\texttt{case3\_sign\_check.pdf}, and the
Credit-Risk figure \texttt{credit\_risk\_profile.pdf}. The console output
reports the plateau errors, trapping residuals, sign margins, damping checks,
and the basis-point conversion used in the financial application.

\end{document}